\documentclass{article}
\usepackage{amsmath}
\usepackage{amsfonts}
\usepackage{amsthm}
\usepackage{amssymb}
\usepackage{color}
\usepackage{enumerate}
\usepackage{tikz}
\usetikzlibrary{patterns}
\usetikzlibrary{decorations.pathreplacing,angles,quotes}
\usetikzlibrary{arrows}
\usetikzlibrary{calc,automata,patterns,decorations,decorations.pathmorphing}
\usetikzlibrary{fadings}
\usepackage{pgfplots}
\pgfplotsset{compat=newest}
\usepgfplotslibrary{fillbetween}
\usepackage{caption}
\usepackage{subcaption}

\usepackage{ulem}  
\usepackage{graphicx}
\usepackage[colorlinks=true]{hyperref}
\hypersetup{urlcolor=blue, citecolor=green, linkcolor=blue}

\newtheorem{theorem}{Theorem}[section]
\newtheorem{lemma}[theorem]{Lemma}
\newtheorem{proposition}[theorem]{Proposition}
\newtheorem{observation}[theorem]{Observation}

\newtheorem{problem}{Problem}

\theoremstyle{definition}
\newtheorem{definition}[theorem]{Definition}
\theoremstyle{remark}
\newtheorem{remark}[theorem]{Remark}

\newcommand\remove[1]{}

\def\ma{\mathcal{A}}
\def\mb{\mathcal{B}}

\def\mj{\mathcal{J}}
\def\mf{\mathcal{F}}
\def\me{\mathcal{E}}

\def\ml{\mathcal{L}}
\def\mM{\mathcal{M}}
\def\mN{\mathcal{N}}
\def\mt{\mathcal{T}}

\def\f2{\mathbb{F}_2}

\def\dist{\hskip0.02cm{\rm dist}\hskip0.01cm}
\def\lip{\hskip0.02cm{\rm Lip}\hskip0.01cm}

\newcommand{\ep}{\varepsilon}
\newcommand{\lin}{{\rm lin}\hskip0.02cm}

\newcommand{\inn}[2]{\langle #1, #2 \rangle}
\newcommand{\lp}[1]{\left( #1 \right)}
\newcommand{\ls}[1]{\left[ #1 \right]}
\newcommand{\lc}[1]{\left\{ #1 \right\}}
\newcommand{\av}[1]{\left| #1 \right|}
\newcommand{\nm}[1]{\left\| #1 \right\|}

\newcommand{\ve}{\varepsilon}

\newcommand{\reo}{\mathbb{R}}

\begin{document}

\title{Dvoretzky-type theorem for locally finite subsets of a Hilbert space}

\author{Florin Catrina, Sofiya Ostrovska, and Mikhail I.~Ostrovskii}

\date{\today}
\maketitle

\noindent{\bf Abstract:} The main result of the paper:  Given any $\ep>0$, every locally finite subset of $\ell_2$
admits a $(1+\ep)$-bilipschitz embedding into an arbitrary
infinite-dimensional Banach space. The result is based on two results which are of independent interest: (1) A direct sum of two finite-dimensional Euclidean spaces contains a sub-sum of a
controlled dimension which is $\ep$-close to a direct sum with respect to a $1$-unconditional basis in a two-dimensional space.
(2) For any finite-dimensional Banach space $Y$ and its direct sum $X$ with itself with respect to a $1$-unconditional basis in a two-dimensional space, there exists a $(1+\ep)$-bilipschitz embedding of $Y$ into $X$ which on a small ball coincides with the identity map onto the first summand and on  the complement of a large ball coincides with the identity map onto the second summand.
\medskip

\noindent{\bf R\'esum\'e.} Le r\'esultat principal de l'article:
\'Etant donn\'e $\ep>0$, chaque sous-ensemble localement fini de
$\ell_2$ admet un plongement $(1+\ep)$-bilipschitz dans n'importe
quel espace de Banach de dimension infinie. Le r\'esultat est
bas\'e sur deux r\'esultats qui pr\'esentent un int\'er\^et
ind\'ependant: (1) Une somme directe de deux espaces euclidiens de
dimension finie contient une sous-somme de dimension contr\^ol\'ee
qui est $\ep$-proche d'une somme directe par rapport \`a une base
$1$-inconditionnelle dans un espace \`a deux dimensions. (2) Pour
tout espace de Banach de dimension finie $Y$ et sa somme directe
$X$ avec lui-m\^eme par rapport \`a une base $1$-inconditionnelle
dans un espace \`a deux dimensions, il existe un plongement
$(1+\ep)$-bilipschitz de $Y$ dans $X$ qui co\"incide, sur une
petite boule, avec l'identit\'e sur la premi\`ere composante, et
qui co\"incide, sur le compl\'ement d'une grosse boule, avec
l'identit\'e sur la deuxi\`eme composante. \medskip

\medskip


\begin{large}


\section{Introduction}
All normed vector spaces considered in this paper are over the reals. 

Recall the classical Dvoretzky Theorem \cite{Dvo59,Dvo61} which
proved Grothendieck's conjecture \cite[Section 7]{Gro53}.

\begin{theorem}[{\cite[Section 7]{Dvo61}}]\label{T:Dvo} Let $k\in\mathbb{N}$, $k\ge 2$,  and $0<\ep<1$.
There exists $N = N (k,\ep)\in\mathbb{N}$ so that every normed
space having more than $N$ dimensions - in particular every
infinite-dimensional normed space - has a $k$-dimensional
subspace whose Banach-Mazur distance from the $k$-dimensional Hilbert space is less than $(1+\ep)$.
\end{theorem}

In this connection, it is natural to call a result establishing
the significant presence of Hilbert space structures in an
arbitrary infinite-dimensional Banach space a {\it Dvoretzky-type
theorem}.\medskip

The following classes of spaces and embeddings are very important in applications, see \cite{BL00, NY12, Ost13}.

Recall that a metric space is called {\it locally finite}
if each ball of finite radius in it contains finitely many elements.  A map  $F: \mM\to\ml$  between two metric spaces $(\mM,d_\mM)$ and $(\ml, d_\ml)$ is called a {\it bilipschitz embedding} if there exist constants $C_1,C_2>0$ so that for all $u,v\in \mM$
\[C_1d_\mM(u,v)\le d_\ml(F(u),F(v))\le C_2d_\mM(u,v).\]
The {\it distortion} of $F$ is defined as
$\lip(F)\cdot\lip(F^{-1}|_{F(\mM)})$, where $\lip(\cdot)$ denotes
the Lipschitz constant. A bilipschitz embedding whose distortion
does not exceed $C\in[1,\infty)$ is called {\it $C$-bilipschitz}.
An embedding satisfying $d_\ml(F(u),F(v))=d_\mM(u,v)$ is called an
{\it isometric embedding}.

A map $F:(\mM,d_\mM)\to (\ml,d_\ml)$ between two metric spaces is called a
{\it coarse embedding} if there exist non-decreasing functions
$\rho_1,\rho_2:[0,\infty)\to[0,\infty)$ such that
$\lim_{t\to\infty}\rho_1(t)=\infty$ and
$$\forall u,v\in \mM~ \rho_1(d_\mM(u,v))\le
d_\ml(F(u),F(v))\le\rho_2(d_\mM(u,v)).$$
\medskip

The main goal of this paper is to prove the following
Dvoretzky-type theorem:

\begin{theorem}\label{T:Main} Given any $\ep>0$, every locally finite subset of $\ell_2$
admits a $(1+\ep)$-bilipschitz embedding into an arbitrary
infinite-dimensional Banach space.\end{theorem}

Note that there exist locally finite subsets of $\ell_2$ which do
not admit isometric embeddings into some infinite-dimensional
Banach spaces, see \cite[Theorem~1.8]{OO19b}.\medskip

At this point, it is appropriate to present a short overview of the available Dvoretzky-type results and related open problems.\medskip

First, we recall the open problem on the validity of a {\it finite isometric Dvoretzky Theorem}  for all infinite-dimensional Banach spaces.

\begin{problem}[\cite{Ost15}, published in \cite{KO18}] Do there exist a finite subset $F$ of $\ell_2$
and an infinite-dimensional Banach space $X$ such that $F$ does
not admit an isometric embedding into $X$?
\end{problem}

A related negative result for spaces $\ell_p$, $1<p<\infty$, $p\ne
2$ was proved in \cite{KO18}.
\medskip

The following weaker version of Theorem \ref{T:Main} was proved in
\cite[Theorem 1]{Ost06}.

\begin{theorem}\label{T:CoarseLF} Each locally finite subset of
$\ell_2$ admits a coarse embedding into an arbitrary
infinite-dimensional Banach space.
\end{theorem}

Using the technique of \cite{BL08}, different from that employed in \cite{Ost06}, Theorem \ref{T:CoarseLF} was strengthened to

\begin{theorem}[{\cite[Theorem 4.3]{Ost09}}]\label{T:BilipLF} Each locally finite subset of
$\ell_2$ admits a bilipschitz embedding into arbitrary
infinite-dimensional Banach space.
\end{theorem}

The upper estimate for the distortion of embeddings of a locally
finite subspace of $\ell_2$ into an arbitrary infinite-dimensional
Banach space obtained in \cite{Ost09} is $100$.
The present paper aims to prove the best possible result in this direction.

As another development, Nowak \cite{Now06} showed that the
embedding techniques of \cite{DG03} can be used to find coarse
embeddings of Hilbert space into Banach spaces for which such an
embeddability appeared to be somewhat unexpected. Later,
Ostrovskii \cite{Ost09} combined the technique of Nowak
\cite{Now06} with the results of \cite{OS94} and strengthened
Nowak's result as follows:

\begin{theorem}[{\cite[Theorem 5.1]{Ost09}}]\label{T:CoarseL2} Let $X$ be a Banach space containing a
subspace with an unconditional basis which does not contain
$\ell_\infty^n$ uniformly. Then $\ell_2$ embeds coarsely into $X$.
\end{theorem}

Theorem \ref{T:CoarseL2} together with Theorem \ref{T:CoarseLF}
led to the problem: {\it Is it true that $\ell_2$ embeds
coarsely into an arbitrary infinite-dimensional Banach
space?} This problem was posed in \cite[pp.~1--2]{Ost06} and published in \cite[Problem 4.1]{Ost09}.
\medskip

A positive answer to this problem would be a significant strengthening of Theorem \ref{T:CoarseLF}, yet, as the matter stands, it was answered in the negative in \cite[Corollary
B]{BLS18}, a typical counterexample is the Tsirelson space
constructed in \cite{Tsi74}.\bigskip

One of the most important directions related to the
Dvoretzky Theorem
is finding optimal estimates for the function $N(k,\ep)$ in the
statement of Theorem \ref{T:Dvo} (see \cite{Mil71}, \cite{MS86},
\cite{Sch06}, \cite{AGM15}, \cite{PV18}, \cite{AGM21}).\medskip

Starting with the paper of Bourgain-Figiel-Milman \cite{BFM86},
a parallel theory for metric spaces was developed. In this theory
the main goal is estimating from below the size - defined  either
as cardinality or in some measure-theoretic ways - of subsets of a
metric space which admit low-distortion embeddings into a Hilbert
space. We list a representative selection of papers devoted to the
results of this type and their applications: \cite{BBM06},
\cite{BLMN05}, \cite{MN07}, \cite{NT12}, \cite{MN13}. See also a
short survey in \cite[Section 8]{Nao12}.
\medskip

Our proof of the main Theorem \ref{T:Main}  will be presented
according to the scheme below:

\begin{itemize}

\item First, an almost-unconditionality result for sums of two
Euclidean spaces will be established in Theorem \ref{T:AUC}.

\item Next, Theorem \ref{T:BendDetail} provides a bending result for
two-dimensional unconditional sums.

\item Finally, combining these results in the spirit of
\cite{OO19}, Theorem \ref{T:Main} will be proved in Section
\ref{S:ConstrEmb}.

\end{itemize}

In addition, we prove a non-bending result, see Theorem
\ref{T:CountGenBendP}. It is related to the
following open problem:

\begin{problem}[{\cite[Problem~5.1]{OO19}}]\label{P:OO19}
Do there exist $\alpha>1$, a locally finite metric space $\mathcal{M}$,
and a Banach space $X$ such that all finite subsets of $\mathcal{M}$ admit
isometric embeddings into $X$, but any bilipschitz embedding of
$\mathcal{M}$ into $X$ has distortion at least $\alpha$?
\end{problem}

We use the standard terminology and notation of Banach space theory \cite{BL00, JL01, LT77, LT79}, local theory \cite{AGM15, AGM21, MS86}, and theory of metric embeddings \cite{Mat02}, \cite{Ost13}.

\section{Almost-unconditionality result}

\begin{definition} Let $Y_1\oplus Y_2$ be a direct sum in which the subspaces $Y_1$ and $Y_2$ are Euclidean, and let $\ep\in[0,1)$. The sum $Y_1\oplus Y_2$ is endowed with a norm whose restrictions to $Y_1$ and $Y_2$ are the Euclidean norms.
We say $Y_1\oplus Y_2$ is {\it $\ep$-invariant} if for any orthogonal operator $O_1$ on $Y_1$ and any orthogonal operator $O_2$ on $Y_2$, the inequality
\begin{equation}\label{E:AlmOrthInv}(1-\ep)\|y_1+y_2\|\le \|O_1y_1+O_2y_2\|\le
(1+\ep)\|y_1+y_2\|\end{equation}
holds.
\end{definition}

As it will be shown below, this invariance is related to unconditionality, see Lemmas \ref{L:InvNorm} and \ref{L:InvToUncond}.

For a direct sum $X=X_1\oplus X_2$ by {\it direct sum projections} we mean projections $P_1:X\to X_1$ and $P_2:X\to X_2$ given by $P(x_1,x_2)=x_1$ and $P(x_1,x_2)=x_2$, respectively.

\begin{theorem}\label{T:AUC} Given $n\in\mathbb{N}$, $\ep\in(0,1),$ and $A\in[1,\infty)$
there exists $N\in \mathbb{N}$, such that, for every direct sum
$X=X_1\oplus X_2$ with both $X_1$ and $X_2$ isometric to $\ell_2^N$,
and the direct sum projections having norms $\le A$,
there are $n$-dimensional subspaces $Y_1\subset X_1$ and $Y_2\subset X_2$, such that the norm on $Y_1\oplus Y_2$ induced from
$X$ is $\ep$-invariant.
\end{theorem}

To see that Theorem \ref{T:AUC} can be understood as an
almost-unconditionality result we need the following two lemmas.

\begin{lemma}\label{L:InvNorm} Let $Y=Y_1\oplus Y_2$ be a direct sum of Euclidean subspaces with an $\ep$-invariant norm $\|\cdot\|$. Let
\[|||y_1+y_2|||=\sup_{O_1,O_2{\rm~orthogonal~on~}
Y_1,Y_2}\|O_1y_1+O_2y_2\|,\quad y_1\in Y_1, y_2\in Y_2.\]

Then $|||\cdot|||$ is a norm on $Y_1\oplus Y_2$ satisfying
\[
\|y_1+y_2\|\le |||y_1+y_2|||\le (1+\ep)\|y_1+y_2\|\] and
\[    |||V_1y_1+V_2y_2|||=|||y_1+y_2|||\]
for every orthogonal operators $V_1$  on $Y_1$ and $V_2$  on $Y_2$.  
Also, the norms $\|\cdot\|$ and $|||\cdot|||$ coincide on $Y_1$ and
$Y_2$. Thus, the norm $|||\cdot|||$ is $0$-invariant on $Y=Y_1\oplus Y_2$.
\end{lemma}

\begin{proof} Proof is straightforward.
\end{proof}

A norm $\|(a,b)\|$ on  $\mathbb{R}^2$ is called {\it $1$-unconditional} if $\|(\pm a,\pm b)\|=\|(a,b)\|$ for every
$(a,b)\in \mathbb{R}^2$.

\begin{lemma} \label{L:InvToUncond} If a norm on a direct sum $Y=Y_1\oplus Y_2$ of two
Euclidean spaces satisfies
\begin{equation}\label{E:RotInv}\|O_1y_1+O_2y_2\|=\|y_1+y_2\|\end{equation}
for all $y_1\in Y_1$, $y_2\in Y_2$ and all orthogonal operators
$O_1$ on $Y_1$ and $O_2$ on $Y_2$, then there exists a
$1$-unconditional norm
$\|\cdot\|_Z$ on $\mathbb{R}^2$ such that
\begin{equation}\label{E:Znorm}
\|y_1+y_2\|=\|(\|y_1\|,\|y_2\|)\|_Z
\end{equation}
\end{lemma}

\begin{proof} If the norm of  
 $Y_1\oplus Y_2=\ell_2^{n_1}\oplus\ell_2^{n_2}$
 satisfies
\eqref{E:RotInv}, we can define a nonnegative function $f$ on the
nonnegative quadrant of $\mathbb{R}^2$ by
\[f(a_1,a_2)=\|y_1+y_2\|,\]
where $y_1\in Y_1$ is such that $\|y_1\|=a_1$ and $y_2\in Y_2$ is
such that $\|y_2\|=a_2$. Equality \eqref{E:RotInv} in combination
with the transitivity of the group of orthogonal operators on any
$0$-centered sphere implies that the resulting function
$f(a_1,a_2)$ is well-defined.

We extend $f$ to $\mathbb{R}^2$ by
\[f(a_1,a_2)=f(|a_1|,|a_2|).\]

It remains to verify that the resulting function $f(a_1,a_2)$ is a
$1$-unconditional norm on $\mathbb{R}^2$.

The only norm property that needs checking is the triangle
inequality since the others are immediate from the definition of
$f$.

Let us  verify the triangle inequality. Clearly,
$$f(a_1+b_1,a_2+b_2)=f(|a_1+b_1|, |a_2+b_2|) =f(\rho_1|a_1|+\sigma_1|b_1|,\rho_2|a_2|+\sigma_2|b_2|),$$ for some $\rho$'s and $\sigma$'s belonging to
the set $\{-1,+1\}$. Hence, taking $u_1\in Y_1$ and $u_2\in Y_2$ to be unit
vectors, one has:
\begin{align*}f(a_1+b_1,a_2+b_2)=
\|(\rho_1|a_1|+\sigma_1|b_1|)u_1+(\rho_2|a_2|+\sigma_2|b_2|)u_2\| \\
\leqslant \|\rho_1|a_1|u_1+\rho_2|a_2|u_2\|+
\|\sigma_1|b_1|u_1+\sigma_2|b_2|u_2\| \\= \|
|a_1|(\rho_1u_1)+|a_2|(\rho_2u_2)\|+
\||b_1|(\sigma_1u_1)+|b_2|(\sigma_2u_2)\|\\=f(|a_1|,
|a_2|)+f(|b_1|,|b_2|)=f(a_1,a_2)+f(b_1,b_2).~~~
\qedhere\end{align*}
\end{proof}

\begin{proof}[Proof of Theorem \ref{T:AUC}] We start by picking $N\in\mathbb{N}$, $\ep>0$, $A\in [1,\infty)$, and a direct
sum $X_1\oplus X_2$ satisfying the conditions of Theorem
\ref{T:AUC}. Our goal is to find $n$ such that the conditions of Theorem \ref{T:AUC} are satisfied, and to establish that $n\to\infty$ as $N\to\infty$.

We will consider two metric structures on $X_1\oplus X_2$. One of
them is induced by the norm of $X$, the other is a Euclidean
structure on $X_1\oplus X_2$ for which $X_1$ and $X_2$ are
orthogonal and have the same Euclidean norms as in $X$.

To find the subspaces $Y_1$ and $Y_2$ for a
given $\ep$ and $A$, we start with
an asymmetric problem. More precisely, for some $x_1\in S(X_1)$ (the unit sphere of $X_1$, it is the same in both norms),
consider the space  $\lin(X_2\cup \{x_1\})$ where $\lin$ denotes the linear span 
of $X_2\cup \{x_1\}$. The ``asymmetric
problem'' to which we refer above is to find a subspace
$E(X_2,x_1)$ of $X_2$ such that the closed unit ball  $B$
(in the norm of
the space $X$) of the space $\lin(E(X_2,x_1)\cup\{x_1\})$ is
{\it $\omega$-invariant with respect to orthogonal operators on the
space $E(X_2,x_1)$}, in the sense that
\[(1-\omega)\|\alpha x_1+y_1\|\le \|\alpha x_1+Oy_1\|\le (1+\omega)\|\alpha x_1+y_1\|\]
for every $\alpha\in\mathbb{R}$, every $y_1 \in E(X_2,x_1)$, and every 
orthogonal operator $O$ on $E(X_2,x_1)$.
A selection of $\omega>0$ needed to get an $\ep$-invariant norm will be
specified later. As the first step in the desired direction, we
observe that an application of \cite[Theorem~7 and Remark~8]{Gor88}, which is a quantitative version for the result of \cite[Corollary of Theorem 2]{LM75}, yields Lemma \ref{L:DeltaSect} below.

By a {\it pointed convex body} in a $k$-dimensional affine space $L$ we mean a pair consisting of a full-dimensional bounded convex body and a point in its interior.
We say that a pointed convex body $(K,z)$ in an affine space $L$ with a Euclidean structure is {\it $\delta$-equivalent $(\delta>0)$ to a Euclidean ball} if there exists $r>0$ such that  the following inclusion holds for Euclidean balls in $L$ centered at $z$:
\[B(z,r)\subset K\subset B(z,(1+\delta)r).\]

\begin{lemma}\label{L:DeltaSect} For any $x_1\in S(X_1)$,  for any $0<\delta<1$, there exists a subspace $E(X_2,x_1)$ of
$X_2$ satisfying the conditions:

\begin{enumerate}[{\bf (1)}]

\item \label{I:ToInfty} Its dimension can be estimated from below
in terms of $N$ (recall that $X_2 =\ell_2^N$) and $\delta$; and this dimension tends to
$\infty$ if $\delta$ is fixed and $N\to\infty$. For  convenience, $\delta$ will be chosen in such a way that $ k:=\frac1\delta\in
\mathbb{N}$.

\item Pointed convex bodies whose components are sections of $B$ by affine subspaces $E(X_2,x_1)+ s\delta
x_1$ and points $s\delta
x_1$, where $s=0,1,\dots,\frac1{\delta}-1,$  are
$\delta$-equivalent to Euclidean balls in the Euclidean structure described above. If $((E(X_2,x_1)+x_1)\cap B, x_1)$ is a pointed convex body in $E(X_2,x_1)+x_1$, it is also  required to satisfy the same condition.
\end{enumerate}

\end{lemma}

\begin{proof} We use \cite[Theorem 7]{Gor88} to construct the subspace $E(X_2,x_1)$ by reducing the space $X_2$ to $E(X_2,x_1)$ in
$k=\frac1{\delta}$ steps. Let step $m$ be such that after this step the
condition of $\delta$-equivalence to Euclidean balls is
satisfied for levels $0,1,\dots,m$ (that is, for subspaces $E(X_2,x_1)+s\delta x_1$ with $s=0,1, \dots, m$).

Observe that the intersection of the ball of $X$ with $X_2$ is a
Euclidean ball, therefore the condition of the item {\bf (2)} for
$m=0$ is satisfied.

After that we start reducing the subspace $X_2$ as follows.

{\bf Step $1$} corresponding to $m=1$. We start with the subspace
$E_0=X_2$ and denote the unit ball of 
 $\lin(E_0\cup\{x_1\})$  in the
$X$-norm by $B_0$. Consider the
intersection of $B_0$ and the affine subspace $\delta x_1+E_0$. Since $\delta<1$, it
is clear that $\delta x_1$ is an interior point of this section 
(recall that $x_1$ is a unit vector in $X_1$).
By \cite[Theorem~7 and Remark~8]{Gor88}, there is a linear subspace $E_1\subset
E_0$ such that the intersection of $B_0$ with $\delta x_1+E_1$ is
$\delta$-equivalent to a Euclidean ball (centered at $\delta
x_1$) and $\dim E_1\ge g(\dim E_0, \delta)$, where  $g$ is given by 
\begin{equation}
\label{E:Gordon}
g(N,\delta)=\delta^2\ln(\sigma N)/\beta
\end{equation}
 for some universal constants $\sigma>0$ and
$0<\beta<\infty$. Step 1 is complete.

Denote by $g^{\{s\}}$ the function obtained as the $s^{\rm th}$
iteration of $g$, that is,
$g^{\{s\}}(N,\delta)=g(g\dots g(g(N,\delta),\delta)\dots,\delta),\delta)$, $s$ times.
\medskip

{\bf Step $m$:} We start with a subspace $E_{m-1}\subset X_2$
whose dimension is at least $g^{\{m-1\}}(N,\delta)$. Denote the unit ball of
$\lin (E_{m-1}\cup\{x_1\})$ in the $X$-norm by $B_{m-1}$.

Note that the intersections of $B_{m-1}$ with the affine subspaces $i\delta
x_1+E_1$ for $i=1,\dots,m-1$ are $\delta$-equivalent to 
Euclidean balls centered at $i\delta x_1$.

Now, consider the intersection of $B_{m-1}$ and
the affine subspace $m\delta x_1+E_{m-1}$. It is clear that
$m\delta x_1$ is an interior point of this section (if
$m<\frac1\delta$).
By \cite[Theorem 7]{Gor88}, there  exists a linear subspace $E_m\subset
E_{m-1}$ such that the intersection of $B_{m-1}$ with $m\delta
x_1+E_m$ is $\delta$-equivalent to the Euclidean ball centered
at $m\delta x_1$ and $\dim E_m\ge g(\dim E_{m-1},\delta)$.

If $x_1$ is an interior point of
 $B_{k-1}\cap(x_1+E_{k-1})$, we stop after
doing Step $k=\frac1\delta$. Otherwise, we stop one step earlier.

We denote the subspace obtained at the end of this procedure by
$E(X_2,x_1)$. It is clear that $\dim E(X_2,x_1)\ge g^{\{k\}}(N,\delta)$.
Since $k$ depends only on $\delta$,  the condition {\bf (1)} of
Lemma  \ref{L:DeltaSect} is satisfied.

It is clear that after this procedure Condition {\bf (2)} is satisfied for all levels, except, 
possibly, level $k$.
\end{proof}

We are going to prove that the established
in Lemma \ref{L:DeltaSect} properties of the ball $B$ 
imply its
$\omega(\delta)$-invariance with respect to orthogonal
operators in $E(X_2,x_1)$, where $\omega(\delta)$ is a function
defined for positive $\delta$ and satisfying
$\lim_{\delta\downarrow0}\omega(\delta)=0$.
\medskip

To do this, we define the function $r$ for $t\in[0,1]$  in the
following manner. Let $B_h(tx_1)$ be the largest Euclidean ball in
the affine subspace $E(X_2,x_1)+tx_1$ centered at $tx_1$ which is
contained in $B$. The  value $r(t)$ is defined 
  to be the radius of
this ball.

Consider the union $C^+:=\displaystyle{\bigcup_{t\in
[0,1]}B_h(tx_1)}$, and let $C^-$ be its image under the central
symmetry about $0$. Since we consider the Euclidean structure in
which $x_1$ is orthogonal to $X_2$ and $B_h(tx_1)$ is a Euclidean
ball centered at $tx_1$, the sets $C^+$ and $C^-$ are reflections of
each other in the subspace $E(X_2,x_1)$. Their union will be denoted by $D$, that is, $D=C^+\cup C^-$. 
The function $r$ is extended as an even function on $[-1,1]$.

\medskip

The following statement holds:

\begin{lemma}\label{L:ConcRad} The function  $r$ is concave and
 continuous on $[-1,1]$, and it is non-increasing on  $[0,1]$.
\end{lemma}

\begin{proof}

Consider  $-1\le t_1<t_2\le 1$.
By the convexity of $B$, the ball $B_h(t_3x_1)$ with $t_3=\alpha t_1+(1-\alpha)t_2$ (for some $0<\alpha<1$) contains the 
ball at level $t_3x_1$ of radius $\alpha r(t_1)+(1-\alpha)r(t_2)$. Therefore, $r(t_3) \ge \alpha r(t_1)+(1-\alpha)r(t_2)$,   
and  $r$  is concave.

The continuity of $r$ on the interval $(-1,1)$ follows from concavity: 
 A function concave on  an interval is continuous  everywhere except,
possibly, at the endpoints of the interval (see, e.g., the introductory
        chapter in \cite{Gru07}).

Monotonicity on $[0,1]$ follows from the concavity of  $r$
 and the fact that it is even on $[-1,1]$.

The continuity of $r$ at $t=1$ (and therefore at $t=-1$ also) follows since $r$ is a  decreasing function bounded
below by $0$  and therefore has a limit $L$ from the left at $t=1$. This
limit coincides with $r(1)$ because $B$ is closed and its intersection with $E(X_2,x_1)+ tx_1$
contains a ball of radius $L$ centered at $tx_1$ for all
$t\in[0,1)$. 
\end{proof}

Next we prove:
\begin{lemma}\label{L:BinsideIncrD} For some $\omega(\delta)>0$ satisfying
$\lim_{\delta\downarrow 0}\omega(\delta)=0$, the inclusion $B \subset (1+\omega(\delta))
D$ holds and thus, $B$ is $\omega(\delta)$-invariant with respect to orthogonal operators on 
$E(X_2,x_1)$. 
\end{lemma}

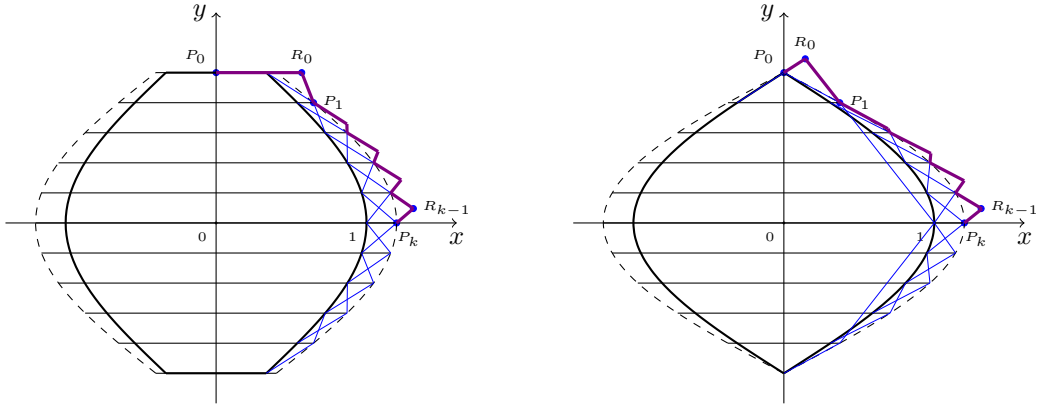
\begin{figure}
\begin{center}
             
\begin{tikzpicture}[scale=0.4]

\node[above left] at (0,5) {\tiny $P_0$};
 \filldraw[blue] (0,5) circle (.1);
 \node[above] at  (5/3*1.707711, 5.000000) {\tiny $R_0$};
  \filldraw[blue] (5/3*1.707711, 5.000000)  circle (.1);
 \node[right] at  ({5/3*1.2*(1+2*cos(1.57/5*(-5+1) r))},{5-1}) {\tiny $P_1$};
  \filldraw[blue]  ({5/3*1.2*(1+2*cos(1.57/5*(-5+1) r))},{5-1}) circle (.1);
 \node[right] at  (5/3*3.9336862, 0.4781379) {\tiny $R_{k-1}$};
   \filldraw[blue] (5/3*3.9336862, 0.4781379) circle (.1);
 \node[below] at  ({6.4},{0}) {\tiny $P_{k}$};
   \filldraw[blue] ({5/3*1.2*(1+2*cos(1.57/5*(-5+5) r))},{5-5}) circle (.1);

 \node[below left] at  (0,0) {\tiny $0$};   
\filldraw  (0,0)  circle (.03);   
   
\node[below left] at  (5,0) {\tiny $1$};   
\filldraw  (5,0)  circle (.03);   
   
\draw[->] (-7,0) -- (8,0) node[below] {$x$};
\draw[->] (0,-6)--(0,7) node[left] {$y$}; 
\draw[-,thick] (-5/3*1,5) -- (5/3*1,5); 
\draw[-,thick] (-5/3*1,-5) -- (5/3*1,-5); 

\draw[thick,domain=-5:5,smooth,variable=\t]
  plot ({5/3*(1+2*cos(1.57/5*\t r))},{\t});
\draw[thick,domain=-5:5,smooth,variable=\t]
  plot ({5/3*(-1-2*cos(1.57/5*\t r))},{\t});
  
 \draw[dashed, thin,domain=-5:5,smooth,variable=\t]
  plot ({5/3*(1.2*(1+2*cos(1.57/5*\t r)))},{\t});
\draw[dashed, thin,domain=-5:5,smooth,variable=\t]
  plot ({5/3*(1.2*(-1-2*cos(1.57/5*\t r)))},{\t});
 
 \foreach \t in {-5,-4,...,5}
\draw[thin] ({5/3*(1.2*(1+2*cos(1.57/5*\t r)))},{\t}) -- ({5/3*(1.2*(-1-2*cos(1.57/5*\t r)))},{\t});  

 \foreach \t in {-5,-4,...,3}
\draw[very thin, blue] ({5/3*((1+2*cos(1.57/5*\t r)))},{\t}) -- ({5/3*(1.2*(1+2*cos(1.57/5*(\t+1) r)))},{\t+1}) ;  
\foreach \t in {-4,-3,...,4}
\draw[very thin, blue] ({5/3*(1+2*cos(1.57/5*(\t+1) r))},{\t+1}) -- ({5/3*1.2*(1+2*cos(1.57/5*\t r))},{\t}) ;  
  

\draw[very thick, violet] (0,5) -- (5/3*1.707711, 5.000000);    
\draw[very thick, violet] ({5/3*1.2*(1+2*cos(1.57/5*(-5+1) r))},{5-1})  -- (5/3*1.707711, 5.000000) ;    
  
\draw[very thick, violet] ({5/3*1.2*(1+2*cos(1.57/5*(-5+1) r))},{5-1})  -- (5/3*2.608540 ,3.291769 ) ;    
\draw[very thick, violet] ({5/3*1.2*(1+2*cos(1.57/5*(-5+2) r))},{5-2})  -- (5/3*2.608540 ,3.291769 ) ;  

\draw[very thick, violet] ({5/3*1.2*(1+2*cos(1.57/5*(-5+2) r))},{5-2})  -- (5/3*3.231391, 2.374698) ;    
\draw[very thick, violet] ({5/3*1.2*(1+2*cos(1.57/5*(-5+3) r))},{5-3})  -- (5/3*3.231391, 2.374698 ) ;  

\draw[very thick, violet] ({5/3*1.2*(1+2*cos(1.57/5*(-5+3) r))},{5-3})  -- (5/3*3.690784, 1.431570) ;    
\draw[very thick, violet] ({5/3*1.2*(1+2*cos(1.57/5*(-5+4) r))},{5-4})  -- (5/3*3.690784, 1.431570) ;  

\draw[very thick, violet] ({5/3*1.2*(1+2*cos(1.57/5*(-5+4) r))},{5-4})  -- (5/3*3.9336862, 0.4781379) ;    
\draw[very thick, violet] ({5/3*1.2*(1+2*cos(1.57/5*(-5+5) r))},{5-5})  -- (5/3*3.9336862, 0.4781379) ;  

\end{tikzpicture}
\hspace{1cm}
\begin{tikzpicture}[scale=0.4]
 
\node[above left] at (0,5) {\tiny $P_0$};
 \filldraw[blue] (0,5) circle (.1);
 
 \node[above] at  (0.711, 5.459) {\tiny $R_0$};
\filldraw[blue] (0.711, 5.459) circle (.1);

 \node[right] at  ({1.2*5*cos(1.57/5*(5-1) r))},{5-1})  {\tiny $P_1$};
  \filldraw[blue]   ({1.2*5*cos(1.57/5*(5-1) r))},{5-1})  circle (.1);
 \node[right] at  (5/3*3.9336862, 0.4781379) {\tiny $R_{k-1}$};
   \filldraw[blue] (5/3*3.9336862, 0.4781379) circle (.1);
 \node[below] at  (6.4,0) {\tiny $P_{k}$};
   \filldraw[blue] ({5/3*1.2*(1+2*cos(1.57/5*(-5+5) r))},{5-5}) circle (.1);

\node[below left] at  (0,0) {\tiny $0$};   
\filldraw  (0,0)  circle (.03);   
   
\node[below left] at  (5,0) {\tiny $1$};   
\filldraw  (5,0)  circle (.03);   
   
\draw[->] (-7,0) -- (8,0) node[below] {$x$};
\draw[->] (0,-6)--(0,7) node[left] {$y$};

\draw[thick,domain=-5:5,smooth,variable=\t]
  plot ({5*cos(1.57/5*\t r)},{\t});
\draw[thick,domain=-5:5,smooth,variable=\t]
  plot ({-5*cos(1.57/5*\t r)},{\t});
  
 \draw[dashed, thin,domain=-5:5,smooth,variable=\t]
  plot ({1.2*5*cos(1.57/5*\t r))},{\t});
\draw[dashed, thin,domain=-5:5,smooth,variable=\t]
  plot ({-1.2*5*cos(1.57/5*\t r))},{\t});
 
 \foreach \t in {-4,-3,...,4}
\draw[thin] ({1.2*5*cos(1.57/5*\t r))},{\t}) -- ({-1.2*5*cos(1.57/5*\t r))},{\t});  

 \draw[very thin, blue] (0.711, 5.459) -- ( 5, 0) ;  
 \draw[very thin, blue] (1.858, -4) -- ( 5, 0) ;  
 \draw[very thin, blue] (0.711, 5.459) -- ( -1.548114, 4) ;  
 \foreach \t in {-5,-4,...,2}
\draw[very thin, blue] ({5*cos(1.57/5*\t r)},{\t}) -- ({1.2*5*cos(1.57/5*(\t+1) r)},{\t+1}) ;  
\foreach \t in {-3,-2,...,4}
\draw[very thin, blue] ({5*cos(1.57/5*(\t+1) r)},{\t+1}) -- ({1.2*5*cos(1.57/5*\t r)},{\t}) ;  
  
\draw[very thick, violet] (0,5)  -- (0.711, 5.459) ;    
\draw[very thick, violet] ({1.2*5*cos(1.57/5*(5-1) r))},{5-1})  -- (0.711, 5.459) ;    
  
\draw[very thick, violet] ({1.2*5*cos(1.57/5*(5-1) r))},{5-1})  -- (3.456277, 3.135875) ;    
\draw[very thick, violet] ({1.2*5*cos(1.57/5*(5-2) r))},{5-2})  -- (3.456277, 3.135875) ;  

\draw[very thick, violet] ({1.2*5*cos(1.57/5*(5-2) r))},{5-2})  -- (4.885178, 2.314469) ;    
\draw[very thick, violet] ({1.2*5*cos(1.57/5*(5-3) r))},{5-3})  -- (4.885178, 2.314469) ;  

\draw[very thick, violet] ({1.2*5*cos(1.57/5*(5-3) r))},{5-3})  -- (5.992737, 1.405467) ;    
\draw[very thick, violet] ({1.2*5*cos(1.57/5*(5-4) r))},{5-4})  -- (5.992737, 1.405467) ;  

\draw[very thick, violet] ({1.2*5*cos(1.57/5*(5-4) r))},{5-4})  -- (6.5857824, 0.4706148) ;    
\draw[very thick, violet] ({1.2*5*cos(1.57/5*(5-5) r))},{5-5})  -- (6.5857824, 0.4706148) ;  

 \end{tikzpicture}
 \caption{Flat top $r_0>0$ (Left) and Sharp top $r_0=0$ (Right)}
\label{F:AC}
\end{center}
\end{figure}

The proof of Lemma~\ref{L:BinsideIncrD} will be given below following two preparatory propositions.  
These propositions will be applied to two-dimensional sections of $B$ and $D$. 
Results about the set $A$ below will be applied to two-dimensional sections of $D$. 

In $\mathbb{R}^2$ endowed with a Cartesian system of coordinates $(x,y)$ consider a closed convex domain $A$ 
symmetric about the coordinate axes and containing the points $(\pm1, 0)$ and $(0,\pm 1) $ on its boundary. 
The boundary of $A$ is represented by the thick black curves in Figure~\ref{F:AC}.
For some natural number $k\ge 5$ let $\delta=1/k$ and denote by $\lc{ (r_i, 1-i\delta) \ | \ 1 \le i \le k}$ 
the coordinates of the points in the first quadrant at the intersection of the boundary of $A$
 with the horizontal lines $y=1-i\delta$, $1 \le i \le k$. 

For $i=0$,  let $r_0$ be the largest $x$-coordinate among the points at the intersection of $A$ and the line $y=1$. 
Note that $r_0$ may be positive if $A$ has a flat top,   or it may be zero. We will distinguish between the two cases. 

Define a polygonal line with vertices alternating from the set of points $P_i$ with coordinates $((1+\delta)r_i, 1-i\delta)$, 
$1\le i\le k$ and the set of points $R_i$ at the intersection of  pairs of lines through $(r_{i-1}, 1-(i-1)\delta)$, 
$((1+\delta)r_i, 1-i\delta)$ and $((1+\delta)r_{i+1}, 1-(i+1)\delta)$, $(r_{i+2}, 1-(i+2)\delta)$, $1 \le i \le k-2$ (see Figures~\ref{F:AC} and \ref{F:slice}). 
Let $P_0$ be the point of coordinates $(0,1)$. 

We define $R_{k-1}$ as the intersection of the line through $(r_{k-1}, -\delta)$ and $(1+\delta, 0)$ 
with the line through $(r_{k-2}, 2\delta)$ and $((1+\delta)r_{k-1}, \delta)$. 

 According to the cases $r_0>0$ and $r_0=0$, we define: 
 \begin{itemize}
 \item[(I)]
 In the case $r_0>0$, 
  let $R_0$ be the intersection of the line through $((1+\delta)r_1, 1-\delta)$ and 
$(r_{2}, 1-2\delta)$ with the line $y= 1$.

 \item[(II)]
  In the case $r_0=0$,
 let $R_0$ be the intersection of the line through $(-r_1, 1-\delta)$ and $(0,1)$
with the line  through $(1, 0)$ and  $((1+\delta)r_1, 1-\delta)$. 

\noindent
{\bf Comment. } This choice of $R_0$ could look artificial, but it works for our goals and makes the computation easier. 
\end{itemize}

Consider the closed domain $C$ bounded by the polygon obtained from the reflections of the polygonal line 
$P_0, R_0, P_1, R_1, P_2, R_2, \dots, P_{k-1}, R_{k-1}, P_k$
about the coordinate axes and about the origin. 

\begin{proposition}\label{P:2Dpoly} 
For $0<\delta< 1/4$,
   the domain $C$ is contained in 
$(1+\omega(\delta))A$  (the homothetic dilation of  $A$) where $\omega(\delta)= 4\delta + \delta^\frac12$.
\end{proposition}
\proof
Because both $A$ and $C$ are symmetric about the coordinate axes, we will focus on the parts of
$A$ and $C$ contained in the first quadrant. 

The proof will proceed in two steps: 
 \begin{itemize}
 \item[(i)] We show that the region of $C$ with points $(x,y)$ with $y\le 1-\delta^\frac12$ is contained in the 
 $(1+\omega(\delta))$ horizontal stretch of $A$, i.e. in the set 
 $$\lc{ \lp{(1+\omega(\delta))x,y} \ | \ (x,y)\in A}.$$ 

 \item[(ii)] We show that the points of $C$ above the line $y= 1-\delta^\frac12$ are covered by the full dilation 
 $(1+\omega(\delta))A$. 
\end{itemize}
The proof of the proposition follows from the convexity of $A$ and from items (i) and (ii).

\begin{figure}
\begin{center}
\begin{tikzpicture}[scale=0.8]
\draw[thick] (9,0) -- (11,0); 
 \filldraw[black](11,0) circle (0.08); 
 \node at (9.7,-0.3) {\small $L_{i+1}$};
 \node at (11,-0.25)  {\small $r_{i+1}$};
 \draw[thick](9,0)--(9,3.5);
 
 \draw[thick](6,3.5)--(9,3.5);

\draw[dotted](9,3.5)--(10,3.5);
\filldraw[black](9,3.5) circle (0.08);
\node at (7, 3.2) {\small $L_i$}; 
\node at (8.8,3.25)  {\small $r_i$};
\filldraw[black](9,7) circle (0.08);

\draw[thick] (6,3.5) -- (6,7);
\filldraw[black](6,7) circle(0.08);
\node at (5.65,6.75)  {\small $r_{i-1}$};

\draw[thick] (2.5,7) -- (6,7); 
\draw[dotted] (6,7)--(6.5,7);

\draw[thick] (2.5,7)--(2.5,10.5)--(1,10.5);


\node at (3.7,6.7) {\small $L_{i-1}$};

 \draw[blue,thick](11,0)--(7.57,12);
 \filldraw[black](10,3.5) circle (0.08);
\filldraw[black](6.503,7) circle (0.08); 
\filldraw[black] (10.5,3.5) circle (0.08);


\filldraw[black](9,7) circle (0.08);

\draw[pattern=north east lines, pattern color=gray](6, 3.5)--(10,3.5)--(9.8328,4.0852)--(6.503,7)--(6,7)--(6,3.5);

\draw[pattern=north east lines, pattern color=yellow] (9,7)--(13.08,7)--(11,12)--(7.57,12)--(9,7);
\draw[pattern=horizontal lines, pattern color=yellow] (6.503,7)--(14.5,0)--(16,0)--(13.08,7)--(6.503,7);
\draw[pattern=crosshatch, pattern color=yellow] (9.8328,4.0852)--(10.5,3.5)--(14.54,3.5)--(13.08,7)--(9,7)-- (9.8328,4.0852);
\draw[pattern=horizontal lines, pattern color=yellow] (10,3.5) --(10.5,3.5)--(9.85,4.05);

\node at (15.3,4.8){\small$P_i: r_i(1+\delta)$};
\node at (13,-0.8){\small$T_i: r_i+2\delta r_{i-1}+(L_{i-1}-L_i)$};
\draw[blue,->](12.5,-0.6)--(10.45,3.4);
\node[above] at (5,10.5){\small$P_{i-1}: r_{i-1}(1+\delta)$};
\draw[->](5,10.7)--(6.5,7.1);

\node at (15,9.5){\small$Q_{i-1}: r_{i-1}+2\delta r_i+(L_i-L_{i+1})$};
\draw[blue,->](14.9,9.3)--(9.1,7.1);

\filldraw[black](10,3.5) circle (0.08);
\draw[blue, thick](14.5,0)--(0.79,12);
\filldraw[black](6.503,7) circle (0.08); 
\filldraw[black] (10.5,3.5) circle (0.08);

\draw[very thick, dashed](6.5,7)--(16,7);
\draw[very thick, dashed](10,3.5)--(16,3.5);

\node[right] at (9.8328,4.0852)  {\small $R_{i-1}$};
 \filldraw[black] (9.8328,4.0852) circle (0.08);

\filldraw[black](9,7) circle (0.08);
\filldraw[black](9,3.5) circle (0.08);
\filldraw[black](6.503,7) circle (0.08);
\filldraw[black](10,3.5) circle (0.08);
\filldraw[black] (10.5,3.5) circle (0.08);
 \filldraw[black](11,0) circle (0.08); 
\draw[->](15.3,4.6)--(10.1,3.55);
\filldraw[black](2.5,10.5) circle (0.08);
\node[left, below]  at (2.1,10.5) {\small$r_{i-2} $};

\node[left]  at (2.5,8.75) {\small$\delta$};
\node[left]  at (6,5.25) {\small$\delta$};
\node[left]  at (9,1.75) {\small$\delta$};

\end{tikzpicture} 
\caption{Analysis at the boundary of $C$.}
\label{F:slice}
\end{center}
\end{figure}

Figure \ref{F:slice} displays:
\begin{itemize}

\item Labels $r_i$ which indicate the points of coordinates
$\lp{r_i,1-i\delta}$ on the boundary of $A$.

\item Horizontal intervals of length $L_i:= r_i-r_{i-1}$ ending at the points $r_i$.

\item Labels $(1+\delta)r_i$  which indicate the points of coordinates
$((1+\delta)r_i, 1-i\delta)$ denoted above by $P_i$.

\item  For $1 \le i\le k-1$,   the number $r_{i-1}+2\delta r_i+(L_i-L_{i+1})$ which represents the $x$-coordinate of the point $Q_{i-1}$
at the intersection of the line through
$((1+\delta)r_{i}, 1-i\delta)$, $(r_{i+1}, 1-(i+1)\delta)$ and the line $y=1-(i-1)\delta$. 
We define the point $Q_{k-1}$ as the intersection of the line through the points 
$(1+\delta, 0)$, $(r_{k-1}, -\delta)$ and the line $y=\delta$;
 it has coordinates $(2+2\delta-r_{k-1},\delta)$. 

\item The number $r_{i}+2\delta r_{i-1}+(L_{i-1}-L_{i})$ which represents the $x$-coordinate of the point $T_{i}$
at the intersection of the line through
$(r_{i-2}, 1-(i-2)\delta)$, 
$((1+\delta)r_{i-1}, 1-(i-1)\delta)$ and the line $y=1-i\delta$. 

\item The point $R_{i-1}$ at the intersection of the line through $P_{i-1}$, $T_i$ and the line through 
$Q_{i-1}$, $P_i$. The intersection looks as is shown in Figures~\ref{F:AC} and \ref{F:slice} because of the monotonicity of $\lc{L_i}_{i=1}^k$ observed below.
\end{itemize}
The points $P_{i-1}, R_{i-1}, P_i$ are consecutive vertices of the polygon that bounds the region $C$ described above. 

Define $L_0:=r_0$.
Observe that the convexity of the domain $A$ implies that $L_1\ge L_2\ge\dots\ge L_k$, while
nothing of this type can be claimed about $L_0$.

To prove step (i) we show that for each $\delta^{-1/2}+1\le i \le k-1$  the following inequalities hold:
\begin{equation}\label{E:B_i-1}
(1+\omega(\delta))r_{i-1}\ge
 r_{i-1}+2\delta r_i+(L_i-L_{i+1}),
 \end{equation}
\begin{equation}\label{E:T_i}
(1+\omega(\delta))r_{i}\ge
 r_{i}+2\delta r_{i-1}+(L_{i-1}-L_{i}).
 \end{equation}

Our proof shows \eqref{E:T_i} for $i=k$ also. We prove a suitable version of \eqref{E:B_i-1} below,  see \eqref{E:k-1}. 

Indeed, if \eqref{E:B_i-1} fails, we get
\[(4\delta+\delta^{1/2})r_{i-1}<2\delta r_i+(L_i-L_{i+1}),\]
or
\[\delta^{1/2}r_{i-1}<4\delta\lp{r_{i}/2-r_{i-1}} +(L_i-L_{i+1}) .\]
For $i\ge 2$, the inequality $L_i\le L_{i-1}$ implies $r_i/2\le r_{i-1}$, 
and thus the last inequality implies
\[\delta^{1/2}r_{i-1} <  (L_i-L_{i+1})\le L_i\le L_{i-1}\le\dots\le
L_1.\]

Therefore $(i-1)\delta^{1/2}r_{i-1}< \sum_{j=1}^{i-1}L_j=r_{i-1}-r_0 \le r_{i-1} $, 
implying $(i-1)<\delta^{-1/2}$, contrary to the
assumption.
\medskip

Meanwhile, if  \eqref{E:T_i} fails, we get
\[(4\delta+\delta^{1/2})r_i<2\delta r_{i-1}+(L_{i-1}-L_{i}).
\]

Since $r_{i-1}\le r_{i}$,  
we get
\[(2\delta+\delta^{1/2})r_{i-1} \le (4\delta+\delta^{1/2})r_i<2\delta r_{i-1}+(L_{i-1}-L_{i}).
\]
This inequality implies
\[\delta^{1/2}r_{i-1} < (L_{i-1}-L_i) \le L_{i-1}\le\dots\le
L_1.\]
As in the previous case, this
implies $(i-1)<\delta^{-1/2}$, contrary to the assumption.

Thus, for $\delta^{-1/2}+1 \le  i\le k-1$ the inequalities \eqref{E:B_i-1} and \eqref{E:T_i} hold, and this implies that 
the points $Q_{i-1}$, $T_i$, and consequently $R_{i-1}$, are covered by the horizontal $(1+\omega(\delta))$ stretch of $A$.

Now we show that the point $Q_{k-1}$ is also covered by the horizontal $(1+\omega(\delta))$ stretch of $A$. 

For  any $\delta < 1/4$ it holds that 
\[ \delta^{1/2}(1-4\delta^{3/2}-\delta) \ge 0.\]
This is equivalent to 
\[(2+4\delta+\delta^{1/2})(1-\delta) \ge 2+2\delta. \]
Since $r_k =1 $ and $r_0 \ge 0$, the convexity of $A$ requires that $r_{k-1}\ge 1-\delta$. 
Therefore, 
\[(2+4\delta+\delta^{1/2})r_{k-1} \ge 2+2\delta \]
and thus
\begin{equation}
\label{E:k-1}
(1+\omega(\delta))r_{k-1} \ge 2+2\delta-r_{k-1},
\end{equation}
which is what we needed to show.

Consequently, for each $\delta^{-1/2}+1 \le i\le k $  the corners $R_{i-1}$, $P_i$ 
of the polygon bounding $C$ are covered by the horizontal $(1+\omega(\delta))$ stretch of $A$. 

We now prove step (ii).  We use the fact that multiplication of $A$
by $(1+\omega(\delta))$ stretches $A$ in all directions, not only
 horizontally.

After the horizontal
stretch of $A$, on level $i$ (i.e., on the line $y=1-i\delta$), we obtain an interval 
of length $(1+\omega(\delta))r_i$. It is easy to see that it
is at least as long as the interval of level $(i-1)$ of length
$r_{i-1}+2\delta r_i+(L_i-L_{i+1})$.
Indeed, 
\[ (1+\omega(\delta))r_i > (1+4\delta)r_i = r_{i-1}+4\delta r_i + r_i- r_{i-1} \ge r_{i-1}+2\delta r_i+(L_i-L_{i+1}).\]
 By the convexity of the image of
$A$, it is enough to cover the interval of this length, provided that
the vertical stretch moves level
$i$ on or above the level $(i-1)$, that is, we need the inequality
\begin{equation}
\label{E:L_lift}
(1+\omega(\delta))(1-i\delta)\ge (1-(i-1)\delta)
\end{equation} 
for each $i$ satisfying $(i-1)<\delta^{-1/2}$.

We assumed that $\delta<\frac14$. With this in mind, $1-(i-1)\delta>1-\delta^{-1/2}\delta=1-\delta^{1/2}\ge
1/2$. Since the fraction $\frac {a-\delta}{a}$ is increasing in $a$ for $a>0$, we have
\[ \frac{1-i\delta}{1-(i-1)\delta} = 
\frac{\lp{1-(i-1)\delta}-\delta}{1-(i-1)\delta} \ge \frac{\frac12-\delta}{\frac12}. \]
To prove \eqref{E:L_lift} it suffices to show that
\[(1+\omega(\delta))\frac{\frac12-\delta}{\frac12} \ge 1.\]
Since $\omega(\delta)>4\delta$, the left-hand side in the last
inequality is at least
\[(1+4\delta)\left(1-2\delta\right)=1+2\delta-8\delta^2>1\] 
if $\delta<\frac14$.

Now it remains to consider points of $C$ which are above level $y=1$.
Note that such points can exist only if $r_0=0$. In this case we have $L_1=r_1>0$. 
We show that the point  $R_0$  (sketched in Figure~\ref{F:AC} (Right)) 
at the intersection of the line through $(-r_1, 1-\delta)$, $(0,1)$ and the line through 
 $(1, 0)$, $((1+\delta)r_1, 1-\delta)$, is covered by  $(1+\omega(\delta))A$. 
  For this it suffices to show that the coordinates of $R_0$ satisfy
\begin{equation}
\label{E:xR0}
(1+\omega(\delta))r_1 \ge x_{R_0},
\end{equation}
and 
\begin{equation}
\label{E:yR0}
(1+\omega(\delta))(1-\delta) \ge y_{R_0}.
\end{equation}
Direct calculation gives that 
$x_{R_0}= r_1\alpha$ and $y_{R_0}= \delta \alpha +1$ where $\alpha = \frac{(1+\delta)r_1 - \delta}{(1-2\delta-\delta^2)r_1+\delta}$. 
Thus, the inequality \eqref{E:xR0} is equivalent to $(1+\omega(\delta)) \ge \alpha$ and the inequality \eqref{E:yR0} is equivalent to 
$(1-\delta)\omega(\delta) \ge \delta(1+\alpha)$. 

If \eqref{E:xR0} holds, then necessarily \eqref{E:yR0} is also true. Indeed, since by direct verification 
$ (1-\delta)\omega(\delta) \ge \delta(2+\omega(\delta))$, 
it follows from $(1+\omega(\delta)) \ge \alpha$ that $(1-\delta)\omega(\delta) \ge \delta(1+\alpha)$
which is equivalent to  \eqref{E:yR0}. Therefore, it will suffice to show that  \eqref{E:xR0} holds. 

Now we check that  \eqref{E:xR0} holds. Note that  $(1+\omega(\delta)) \ge \alpha$  is equivalent to 
\begin{equation}
\label{E:nar}
\omega(\delta) \ge \delta\frac{(3+\delta)r_1-2}{(1-2\delta-\delta^2)r_1+\delta},
 \end{equation}
 which leads to the following two cases: 
 
 (I) the case $(3+\delta)r_1\le 2$ when the inequality \eqref{E:nar} holds trivially, and 
 
 (II) the case $(3+\delta)r_1> 2$, which means $r_1 >\frac{2}{3+\delta}$. 
 
 \noindent
 Since in case (II), if we increase $r_1$ to $1$ in the numerator of the fraction on the 
  right hand side and decrease $r_1$  to  $\frac{2}{3+\delta}$  in the denominator we obtain 
  \[ \delta\frac{3+4\delta+\delta^2}{2-\delta-\delta^2} \ge \delta\frac{(3+\delta)r_1-2}{(1-2\delta-\delta^2)r_1+\delta}.\]
  Since $0< \delta< 1/4$, a direct check shows that  $4\delta $ is larger than the left hand side of the inequality above. 
 Thus,  since $\omega(\delta)>4\delta $, the inequality  \eqref{E:nar} holds  in both cases.

Therefore, we obtain that the portion of $C$ above level $1$  
is contained inside the $(1+\omega(\delta))$ homothetic image of $A$.
\endproof

Define a family ${\mathcal I}$ of  $4(k-1)$ closed horizontal intervals, which consists of the intervals with endpoints 
$(r_i, 1-i\delta)$ and $((1+\delta)r_i, 1-i\delta)$, $1 \le i \le k-1$, 
together with their reflection about the coordinate axes and about the origin.

\begin{proposition}\label{P:2D} 
Any closed convex domain $H$ which contains $A$, 
and whose boundary  intersects all the intervals in ${\mathcal I}$ and  passes through the points $(\pm1, 0)$, 
$(0,\pm 1) $, is contained in 
$(1+\omega(\delta))A$  (the homothetic dilation of  $A$) where $\omega(\delta)= 4\delta + \delta^\frac12$.
\end{proposition}
\proof 
From the requirement that the domain $H$ is convex it follows that $H$ is contained in the polygonal domain 
$C$ described in Proposition~\ref{P:2Dpoly}. Since $C$  is contained in 
$(1+\omega(\delta))A$, it follows that so is $H$. 
\endproof

\proof[Proof of Lemma~\ref{L:BinsideIncrD}] 
Consider any unit vector $x_2$  in $E(X_2,x_1)$. In the plane spanned by $x_1$ and $x_2$ we consider the 
Cartesian system of coordinates with the $x$-axis in the direction of $x_2$ and the $y$-axis in the direction of $x_1$. 
Define the domains $A$ and $H$ as the respective intersections of $D$ and  the unit ball $B$ with this plane. 
The hypotheses of Proposition~\ref{P:2D} are satisfied and therefore $H$  is contained in 
$(1+\omega(\delta))A$.
Since this holds for any choice of $x_2$, the inclusion $B\subset (1+\omega(\delta))D$ follows. 
Because $D\subset B$ and $D$ is invariant under any orthogonal map on $E(X_2,x_1)$ we conclude that $B$ 
is $\omega(\delta)$-invariant. 
\endproof

Let $G$ be a function, $G:\mathbb{N}\times
(0,\infty)\to\mathbb{N}\cup\{0\}$. We say that $G$ is {\it
indefinitely growing} (IG) if $\lim_{N\to\infty}G(N,\delta)=\infty$ for
every $\delta>0$. 

\begin{observation}\label{O:IterGI}
Finite iterations of {\rm IG} functions with fixed $\delta>0$, that is,
iterations of the form
\[ G(G(\dots G(G(N,\delta),\delta)\dots,\delta),\delta),
\]
are also {\rm IG} functions.
\end{observation}

\begin{proposition}\label{P:InTermsGI} There exists an {\rm IG} function $G(=G(N,\alpha))$
such that, for each subspace $U$ of $X_2$, each $\alpha>0$, and each
$v\in S(X_1)$, there exists a subspace $E(U,v)\subset U$ with $\dim
E(U,v)\ge G(\dim U,\alpha)$ and such that the $X$-norm on
$\lin(E(U,v)\cup\{v\})$ is $\alpha$-invariant with respect to the
orthogonal operators on $E(U,v)$.
\end{proposition}
\proof
For the given $\alpha> 0$ select $\delta > 0$ so that $k:= \frac{1}{\delta} \in \mathbb{N}$ and $\omega(\delta)< \alpha$. 
Applying Lemmas \ref{L:DeltaSect} and \ref{L:BinsideIncrD}  for this value of $\delta$, $X_2=U$ and $x_1=v$, we obtain the subspace 
$E(U,v)\subset U$ with dimension bounded below by $g^{\lc{k}}(\dim U,\delta)$, the $k^{\rm th}$ iteration of the  function \eqref{E:Gordon},
such that   the $X$-norm on
$\lin(E(U,v)\cup\{v\})$ is $\alpha$-invariant with respect to the
orthogonal operators on $E(U,v)$.

Since we may assume that $\delta =\frac1k$ satisfies both $\omega(\delta)< \alpha$ and $\omega\lp{\frac1{k-1}} \ge \alpha$ we can regard 
$g^{\lc{k}}(\dim U,\delta)$ as $G(\dim U, \alpha)$ for some IG function $G$.
\endproof

Combining Observation~\ref{O:IterGI} and
Proposition~\ref{P:InTermsGI}  with the known fact that the
cardinality of an $\alpha$-net in the unit sphere of a
$t$-dimensional normed space can be estimated from above in terms
of $t$ and $\alpha>0$ only, we arrive at the following statements.

\begin{enumerate}[{\bf (A)}]

\item For every  $\alpha>0$ and $n, M\in \mathbb{N}$, there exists $N\in
\mathbb{N}$ such that if we apply Proposition
\ref{P:InTermsGI} to $X_1$ with $\dim X_1=N$ and all points $x_2$
in an $\alpha$-net $N_2$ of the unit sphere of an $M$-dimensional
subspace $U\subset X_2$, we obtain a subspace $Y_1$ in $X_1$ of
dimension at least $n$, such that
\[(1-\alpha)\|y_1+y_2\|\le \|O_1y_1+y_2\|\le (1+\alpha)\|y_1+y_2\|\]
for every $y_1\in Y_1$,
$y_2$ being any scalar multiple of an element of $x_2\in N_2$, and any orthogonal operator $O_1$ on $Y_1$.

\item For every $\alpha>0$ and $n\in \mathbb{N}$, there exists $M\in
\mathbb{N}$ such that applying Proposition
\ref{P:InTermsGI} to a subspace $U\subset X_2$ of dimension $\dim
U\ge M$ and all points $x_1$ in an $\alpha$-net $N_1$ of the unit
sphere of an $n$-dimensional subspace $Y_1\subset X_1$, brings out a
subspace $Y_2\subset U\subset X_2$ with $\dim Y_2\ge n$ such
that
\[(1-\alpha)\|y_1+y_2\|\le \|y_1+O_2y_2\|\le (1+\alpha)\|y_1+y_2\|\] for every $y_1$ 
being a scalar multiple of an element $x_1\in N_1$,
any  $y_2\in Y_2$, and any orthogonal operator $O_2$ on $Y_2$.

\end{enumerate}

We use items {\bf (A)} and {\bf (B)} as follows. First we use $n$
to find values of $M$ and later $N$. Thereafter, we pick any $N$-dimensional subspaces $X_1$ and $X_2$
satisfying the conditions of Theorem \ref{T:AUC}.

After that we apply item {\bf (A)} to an arbitrarily chosen
$M$-dimensional subspace $U\subset X_2$, and get a subspace
$Y_1\subset X_1$.

Finally, with the help of item {\bf (B)}, for the chosen $U$ and $Y_1$
constructed in the previous step we obtain  $Y_2$.

To conclude the proof of Theorem \ref{T:AUC} we need the following approximation
lemma for each step of the construction. Let
\[A=\max\{\|P_1\|,\|P_2\|\},\]
where $P_1:X_1\oplus X_2\to X_1$ and  $P_2:X_1\oplus X_2\to X_2$
are projections with kernels $X_2$ and $X_1$, respectively, and
the norm is the $X$-norm.

\begin{lemma}\label{L:NetToEvery} The conditions of items {\bf (A)} and {\bf (B)} imply that, for any $y_1\in Y_1$, $y_2\in Y_2$, and any orthogonal operators $O_1$ on $Y_1$ and $O_2$ on $Y_2$, we have

\begin{equation}\label{E:Approx}\begin{split}
(1-\alpha(1+(2-\alpha)A))^2\|y_1+y_2\|&\le
\|O_1y_1+O_2y_2\|\\&\le (1+\alpha(1+(2+\alpha)A))^2\|y_1+y_2\|,
\end{split}
\end{equation}
provided $(1-\alpha(1+(2-\alpha)A))>0$.
\end{lemma}

\begin{proof} We may assume that $y_1\neq 0$ and $y_2\neq 0$.
Let $z_2$ be a multiple of an element from $N_2$ such that $\|z_2-y_2\|<\alpha\|y_2\|$. Then
\begin{equation}\label{E:above}\begin{split}\|O_1y_1+y_2\|&\le \|O_1y_1+z_2\|+\|z_2-y_2\|
\\&\le(1+\alpha)\|y_1+z_2\|+\|z_2-y_2\|
\\&\le(1+\alpha)\|y_1+y_2\|+(2+\alpha)\|z_2-y_2\|
\\&<(1+\alpha)\|y_1+y_2\|+(2+\alpha)\alpha\|y_2\|
\\&\le(1+\alpha(1+(2+\alpha)A))\|y_1+y_2\|.
\end{split}\end{equation}
Similarly (assume that $\alpha<1$),
\begin{equation}\label{E:below}\begin{split}\|O_1y_1+y_2\|&\ge \|O_1y_1+z_2\|-\|z_2-y_2\|
\\&\ge(1-\alpha)\|y_1+z_2\|-\|z_2-y_2\|
\\&\ge(1-\alpha)\|y_1+y_2\|-(2-\alpha)\|z_2-y_2\|
\\&>(1-\alpha)\|y_1+y_2\|-(2-\alpha)\alpha\|y_2\|
\\&\ge(1-\alpha(1+(2-\alpha)A))\|y_1+y_2\|.
\end{split}\end{equation}

Now we apply the same argument for any $w_1\in Y_1$ and any $z_1$ in the direction of an element of $N_1$
 satisfying $\|w_1-z_1\|<\alpha\|w_1\|$.
We get
\begin{equation}\label{E:SecondPart}\begin{split}
(1-\alpha(1+(2-\alpha)A))\|w_1+y_2\|&\le \|w_1+O_2y_2\|\\&\le (1+\alpha(1+(2+\alpha)A))\|w_1+y_2\|.
\end{split}
\end{equation}

Plugging $w_1=O_1y_1$ and using \eqref{E:above}, \eqref{E:below}, and $(1-\alpha(1+(2-\alpha)A))>0$, we get \eqref{E:Approx}.\end{proof}

To complete the proof of Theorem \ref{T:AUC} we pick $\alpha>0$ in such a way that
$(1+\alpha(1+(2+\alpha)A))^2<1+\ep$
and $(1-\alpha(1+(2-\alpha)A))^2>1-\ep$. \end{proof}

\section{Bending in unconditional sums of two
spaces}\label{S:Bend}

Let $X$ and $Y$ be (possibly finite-dimensional) Banach spaces
such that there exist two linear isometric embeddings $I_1:Y\to X$
and $I_2:Y\to X$ with distinct images $Y_1=I_1(Y)$ and
$Y_2=I_2(Y)$.

\begin{definition}\label{D:bend} Let $C\in[1,\infty)$. A mapping
$T:Y\to X$ is called a {\it $C$-bending} of $Y$ in the space $X$
from $I_1$ to $I_2$, with parameters $(r,R)$, $0<r<R<\infty$, if
it is a $C$-bilipschitz embedding  such that the restriction of
$T$ to the ball of radius $r$ coincides with $I_1$ and the
restriction of $T$ to the exterior of the ball of radius $R$ in
$Y$ coincides with $I_2$.
\end{definition}

Let $Z=(\mathbb{R}^2,\|\cdot\|_Z)$ be a two-dimensional Banach
space in which the unit vectors $(1,0)$ and $(0,1)$ form a
normalized $1$-unconditional basis. This means
\begin{equation}\label{E:NormUC}\|(1,0)\|_Z=\|(0,1)\|_Z=1 \hbox{ and }\|(a,b)\|_Z=\|(\pm a,\pm b)\|_Z.\end{equation}

Given a Banach space $Y$, we use $X= Y\oplus_Z Y$ to denote the
Banach space consisting of pairs $(u,v)$ with $u,v\in Y$ with the
norm
\[\|(u,v)\|_X=\|(\|u\|_Y,\|v\|_Y)\|_Z.\]

When we consider a $C$-bending of $Y$ in the space $X=Y\oplus_Z Y$
we restrict our attention to the case where $I_1(y)=(y,0)$ and
$I_2(y)=(0,y)$ and call such bending {\it a $C$-bending of $Y$ in
the space  $X=Y\oplus_Z Y$ with parameters $(r,R)$,
$0<r<R<\infty$.}
\medskip

To state the main result of this section,  Theorem~\ref{T:BendDetail}, we need to
introduce some additional parameters. Define
\[m_Z=\min_\tau \|(\cos \tau,\sin \tau)\|_Z\hbox{ and }
M_Z=\max_\tau \|(\cos \tau,\sin \tau)\|_Z.\]

\begin{observation}\label{O:mM}
It is easy to see that the unit ball of $Z$ satisfying
\eqref{E:NormUC} contains the unit ball of $\ell_1^2$ and is
contained in the unit ball of $\ell_\infty^2$, thereby
$m_Z\ge\frac1{\sqrt{2}}$ and $M_Z\le\sqrt{2}$.
\end{observation}

Let
\begin{equation}
    \label{E:u}
u(\tau)= \frac{(\cos\tau, \sin \tau)}{\nm{(\cos\tau, \sin
\tau)}_Z}\in Z.
\end{equation}
Condition \eqref{E:NormUC} implies that $u(0) = (1,0)$ and
$u(\pi/2) = (0,1)$. We need the following

\begin{proposition}
\label{P:diff} The set of all quotients
\[\frac{\nm{u(\tau_2)-u(\tau_1)}_Z}{\tau_2-\tau_1}\]
for $0\le \tau_1<\tau_2\le \frac{\pi}2$ is bounded. Let
\begin{equation}\label{E:c_Z}
    c_Z:=\sup_{0\le \tau_1<\tau_2\le \frac{\pi}2}\frac{\nm{u(\tau_2)-u(\tau_1)}_Z}{\tau_2-\tau_1}.
\end{equation}
Then
\begin{equation}\label{E:Estc_Z}
\frac{2}{\pi} \le \frac{2\sqrt{2} m_Z}{\pi} \le c_Z\le\frac{2M_Z}{m_Z}\le 4.
    \end{equation}
\end{proposition}
The last inequality follows from Observation \ref{O:mM}. Since in
this paper we do not need tight estimates for $c_Z$, we do not
dwell on their evaluation.

\proof To begin with, we write:
\[\begin{aligned}
u(\tau_2)-u(\tau_1) = &
\frac{(\cos\tau_2 -\cos\tau_1, \sin \tau_2 - \sin \tau_1)}{\nm{(\cos\tau_2, \sin \tau_2)}_Z} \\
 & -(\cos\tau_1, \sin \tau_1)
 \frac{\nm{(\cos\tau_2, \sin \tau_2)}_Z -
\nm{(\cos\tau_1, \sin \tau_1)}_Z}{ \nm{(\cos\tau_1, \sin
\tau_1)}_Z\nm{(\cos\tau_2, \sin \tau_2)}_Z}.
\end{aligned}\]

Applying the triangle inequality to the numerator of the norm of
the second term in the right-hand side, we conclude that the norm
of the second term does not exceed the norm of the first term.
Therefore,
\[ \nm{u(\tau_2)-u(\tau_1)}_Z \le
\frac {2}{m_Z} \nm{(\cos\tau_2 -\cos\tau_1, \sin \tau_2 - \sin
\tau_1)}_Z.\] Trigonometric identities imply the following vector
version of a spherical Mean Value Theorem
\[(\cos\tau_2 -\cos\tau_1,
\sin \tau_2 - \sin \tau_1) = \lp{-\sin \frac{\tau_1+\tau_2}{2},
\cos \frac{\tau_1+\tau_2}{2}} 2\sin \frac{\tau_2-\tau_1}{2}. \]
Therefore
\[\nm{(\cos\tau_2 -\cos\tau_1,
\sin \tau_2 - \sin \tau_1)}_Z \le M_Z (\tau_2-\tau_1), \] and the
inequality $\displaystyle{c_Z\le \frac{2M_Z}{m_Z}}$ follows.

To get the  bound from below on $c_Z$ in \eqref{E:Estc_Z} we substitute 
$\tau_1=0$ and $\tau_2=\frac{\pi}2$  in the quotient in \eqref{E:c_Z}.
\endproof

The main result of this section is the following theorem.

\begin{theorem}\label{T:BendDetail} Let $Y$ be a finite-dimensional Banach space, and let $Z$ be a $2$-dimensional space satisfying \eqref{E:NormUC}. Then for every $\ep>0$ and every pair $(r,R)$ of positive numbers satisfying the condition
\begin{equation}\label{E:CondR}
\frac{\ep}{c_Z}\,\ln\left(\frac Rr\right)=\frac{\pi}{2},
\end{equation}
there is a
$\displaystyle{\left(\frac{1+\ep}{1-\ep}\right)}$-bending $T$ of
$Y$ into the sum $X=Y\oplus_Z Y$ with parameters $(r,R)$.
Furthermore, the bending $T$ satisfies
\begin{equation}\label{E:PresNorm} \|Tx\|_X=\|x\|_Y \ \ \ \mbox{ for all } x\in Y,
\end{equation}
and
\begin{equation}\label{E:ExactForm}
(1-\ep)  \|x-y\|_Y\le \|Tx-Ty\|_{X}\le (1+\ep)\|x-y\|_Y
 \ \ \ \mbox{ for all } x, y \in Y.
\end{equation}
\end{theorem}

\begin{remark}\label{R:OnR}
Any $C$-bending with parameters $(r,R)$ is also a $C$-bending with
parameters $(r_1,R_1)$ if $0<r_1\le r<R\le R_1<\infty$. For this
reason, the exact value of $c_Z$ is not important.
\end{remark}

\begin{proof}[Proof of Theorem \ref{T:BendDetail}] We follow the construction in \cite[Section 2.2]{OO19}.\medskip

Let $\ep\in(0,1), r>0$ be any numbers. For real numbers $t\ge r$,
define the function
\begin{equation}
    \label{E:tau}
    \tau(t)= \tau_{\ve,r,Z}(t):=\frac{\ep}{c_Z}\ln \lp{\frac tr},
\end{equation}
where $c_Z$ is defined in Proposition~\ref{P:diff}. The function
$\tau(t)$ is increasing and, by \eqref{E:CondR}, maps the interval
$[r, R]$ onto $[0,\pi/2]$. The Mean Value Theorem implies that
\begin{equation}
    \label{E:tauDiff}
\tau(t_2)-\tau(t_1) \le \frac{\ep}{c_Z} \frac{t_2-t_1}{t_1}
\end{equation}
for $r \le t_1 \le t_2 \le R$. \remove{Extend $\tau$ to
$[0,\infty)$ by letting $\tau(t)=0$ for $0\le t\le r$ and
$\tau(t)=\pi/2$ for $t \ge R$.}

We introduce the functions
\begin{equation}\label{E:c}
c(x)=\begin{cases} 1&\hbox{ if }\|x\|\le r\\
\frac{\cos \tau(\nm{x})}{\|\lp{\cos \tau(\nm{x}), \sin
\tau(\nm{x})}\|_Z}
&\hbox{ if } r\le \|x\|\le R\\
0 &\hbox{ if }\|x\|\ge R,\\
\end{cases}
\end{equation}
and
\begin{equation}\label{E:s}
s(x)=\begin{cases} 0 &\hbox{ if }\|x\|\le r\\
\frac{\sin \tau(\nm{x})}{\|\lp{\cos \tau(\nm{x}), \sin
\tau(\nm{x})}\|_Z}
&\hbox{ if } r\le \|x\|\le R\\
1 &\hbox{ if }\|x\|\ge R.\\
\end{cases}
\end{equation}
It is clear that \begin{equation}\label{E:Normaliz}
    \|(c(x),s(x))\|_Z=1
\end{equation}
and
\begin{equation}
    \label{E:U}
(c(x),s(x))= u\lp{\tau(\nm{x}_Y)}
\end{equation}
for every $x\in Y$ with $r\le \|x\|_Y\le R$.

We claim that the desired bending is the map
\[ T:Y\to X=Y\oplus_Z Y\]
given by
\begin{equation}\label{E:DefBend} Tx=(c(x)x,s(x)x).\end{equation}

\begin{remark}
It is worth mentioning that $T$ is a development of the well-known in geometry logarithmic spirals in the plane, see \cite[p.~4]{BS07}.
\end{remark}
\smallskip
Equation \eqref{E:Normaliz}  implies that $\|Tx\|_X=\|x\|_Y$ for
every $x\in Y$.
  It is also clear that $T$ satisfies the  condition \eqref{E:ExactForm} whenever $x$, $y$ are both in  the ball of
radius $r$ or in the exterior of the ball of radius $R$.

When estimating $ \nm{Tx-Ty}_X$, from now on we assume without
loss of generality that
\begin{equation}
\label{E:ord}
    \nm{x}_Y \ge \nm{y}_Y.
\end{equation}

Next, we write
\[
 Tx-Ty=   (c(x)x,s(x)x)-(c(y)y,s(y)y)\]
in the form
\begin{equation}\label{E:*} \begin{split}
 Tx-Ty = & (c(x)(x-y),s(x)(x-y)) \\
 & + \lp{ (c(x)-c(y))y,(s(x)-s(y))y}.
 \end{split}\end{equation}
For the first summand in the right-hand side of \eqref{E:*}, we
have:
\[\nm{(c(x)(x-y),s(x)(x-y)) }_X
=\|(c(x)\|x-y\|_Y, s(x)\|x-y\|_Y)\|_Z.\]
We conclude that
\begin{equation}\label{E:**}\nm{(c(x)(x-y),s(x)(x-y)) }_X
=\|x-y\|_Y\|(c(x),s(x))\|_Z=\|x-y\|_Y.
\end{equation}
For the second summand in
the right-hand side of \eqref{E:*}, there holds
\begin{equation}\label{E:***}\nm{ \lp{(c(x)-c(y))y,(s(x)-s(y))y}}_X =
\nm{y}_Y \nm{ (\av{c(x)-c(y)},\av{s(x)-s(y)})}_Z .  
\end{equation}

For $x\in Y$, set
\[U(x) := (c(x),s(x)) \in Z. \]
According to \eqref{E:c}, \eqref{E:s}, and \eqref{E:U}, we have
\begin{equation}
    \label{E:obs}
    U(x) = \begin{cases} u(\tau(r)) &\hbox{ if } \|x\|_Y\le r,\\
    u(\tau(\nm{x}_Y)) &\hbox{ if }r\le \|x\|_Y\le R,\\
    u(\tau(R)) &\hbox{ if }\|x\|_Y\ge R.\end{cases}
\end{equation}

Combining the definition of $U$ with \eqref{E:*}, \eqref{E:**},
and \eqref{E:***},   for any $x,y \in Y$ we obtain
\begin{equation}
\label{E:Dest}
 \begin{split}
 &  \|x-y\|_Y - \nm{y}_Y \nm{ U(x)-U(y)}_Z
\le \nm{ Tx-Ty}_X \\
& \le \|x-y\|_Y +\nm{y}_Y \nm{ U(x)-U(y)}_Z.
 \end{split}
 \end{equation}

Now, we  show that for any $x,y \in Y$ satisfying \eqref{E:ord} we have
\begin{equation}
\label{E:bound_norm} \nm{y}_Y \nm{U(x)-U(y)}_Z\le \ve\nm{x-y}_Y.
\end{equation}
This inequality, together with \eqref{E:Dest} immediately implies
\eqref{E:ExactForm}, and, thus, concludes the proof of the
theorem.

We prove a stronger version of  \eqref{E:bound_norm}, namely,  
\begin{equation}
\label{E:str}
\nm{U(x)-U(y)}_Z \le \ve
\frac{\nm{x}_Y-\nm{y}_Y}{\nm{y}_Y}.
 \end{equation}

It is clear that combining \eqref{E:str} with the triangle inequality we obtain \eqref{E:bound_norm}.
On the other hand, 
\[ \begin{aligned}
& \nm{U(x)-U(y)}_Z   \stackrel{\eqref{E:obs}}{=} 
\nm{u \lp{\tau(\min\lc{R,\nm{x}_Y})}- 
u\lp{\tau(\max\lc{ r,\nm{y}_Y})}}\\
& 
 \stackrel{\eqref{E:c_Z}}{\le} 
 c_Z\lp{\tau(\min\lc{R,\nm{x}_Y})- 
\tau(\max\lc{ r,\nm{y}_Y})} 
 \stackrel{\eqref{E:tauDiff}}{\le} 
 \ve
  \frac{\min\lc{R,\nm{x}_Y}- \max\lc{ r,\nm{y}_Y}}{\max\lc{ r,\nm{y}_Y}}\\
  & \le \ve
\frac{\nm{x}_Y-\nm{y}_Y}{\nm{y}_Y}.
  \end{aligned}
  \]
This
completes the proof of Theorem \ref{T:BendDetail}. \end{proof}

\section{Construction of the embedding}\label{S:ConstrEmb}

\begin{proof}[Proof of Theorem \ref{T:Main}]
Let $X$ be an infinite-dimensional Banach space, $\mathcal{M}$
be a locally finite subset in $\ell_2$, and $\ep\in(0,1)$. We
assume that $0\in \mathcal{M}$. Our goal is to find an embedding
of $\mathcal{M}$ into $X$ with distortion $\le 1+\ep$.\medskip

To achieve the distortion $(1+\ep)$ we need to introduce
additional parameters $\gamma,\psi,\zeta\in (0,1)$, and
$d\in\mathbb{N}$, such that the maximal quotient  of the
right-hand sides and respective left-hand sides in  \eqref{E:Case1}, \eqref{E:Case2},
\eqref{E:EpEstFin}  does
not exceed $(1+\ep)$. Such values exist because the values of all
coefficients go to $1$ as $\gamma,\psi,\zeta\downarrow 0$ and
$d\to\infty$. Also,    we introduce a decreasing sequence
$\{\gamma_i\}_{i=1}^\infty$, $\gamma_i>0$, such that
$\prod_{i=1}^\infty (1+\gamma_i)<1+\gamma.$
\medskip

Next, we define recursively an increasing sequence
$\{R_i\}_{i=1}^\infty$ of positive numbers as follows:

\begin{enumerate}[{\rm (i)}]

\item \quad $R_1=1$.

\item \label{I:2} \quad $\displaystyle{\frac{\psi}4\,
\ln\frac{R_{2i}}{R_{2i-1}}=\frac\pi2}$ for all $i\in \mathbb{N}$.

{\bf Note:} We use the number $4$ in this formula because it is
our upper estimate for $c_Z$ which works for every $Z$, see
\eqref{E:Estc_Z}.

\item \label{I:3} \quad
$\displaystyle{\frac{R_{2i+1}}{R_{2i}}=\frac d\ep}$ for all $i\in \mathbb{N}$.

\end{enumerate}

Let $\mb(R)\subset\mM$ denote the ball of radius $R$ centered at
$0$, while $F_i$ denotes the subspace of $\ell_2$ spanned by
$\mb(R_{4i})$ and $n_i=\dim F_i$.

To prove Theorem \ref{T:Main} we need the following lemma about
FDDs (finite-dimensional Schauder decompositions) in an arbitrary
infinite-dimen\-sio\-nal Banach space $X$. See \cite[p.~11]{JL01}
or \cite[Section 1.g]{LT77} for a basic information on FDDs.

\begin{lemma}\label{L:3rdFDD} Let $\{\gamma_i\}_{i=1}^\infty$ and $\zeta$ be the numbers chosen at the beginning of Section \ref{S:ConstrEmb},
and $\{F_i\}_{i=1}^\infty$ be the subspaces chosen above.
Introduce the sequence $\{\mu_i\}_{i=1}^\infty$ by $\mu_1=\gamma_1$,
$\mu_{2i}=\mu_{2i+1}=\gamma_{i+1}$. 

There exists an infinite-dimensional subspace $V\subset X$ having
an FDD $\{V_i\}_{i=1}^\infty$ for which there exist isomorphisms
$J_i:F_{j(i)}\to V_i$, such that
\[
j(i)=\begin{cases} (i+1)/2 &\hbox{ if $i$ is odd;}\\
i/2 &\hbox{ if $i$ is even,}
\end{cases}
\]
(that is, for each $j\in\mathbb{N}$ there are two isomorphisms
$J_i$ with domain $F_j$) with the following properties

\begin{enumerate}[{\bf (a)}]
    \item $\forall v\in F_{j(i)}\quad \|v\|_2\le \|J_iv\|_X\le (1+\mu_i)\|v\|_2$.

    \item There exist $1$-unconditional norms $\nm{\ }_{Z_i}$ on $\reo^2$ such that the maps 
    $\mj_{2i-1,2i} : F_{j(2i-1)} \oplus_{Z_i} F_{j(2i)}\to V_{2i-1}\oplus V_{2i}$ 
    given by
    $\mj_{2i-1,2i}(u,v)=(J_{2i-1}u,J_{2i}v)$ satisfy   
    \[\|J_{2i-1}u+J_{2i}v\|_X\le \|(\|u\|_2,\|v\|_2)\|_{Z_i}\le
    (1+\zeta)(1+\gamma_i)^2\|J_{2i-1}u+J_{2i}v\|_X
    \]

     \item The maps $\mj_{2i,2i+1}:F_{j(2i)}\oplus_2 F_{j(2i+1)}\to V_{2i}\oplus V_{2i+1}$ given by
    $\mj_{2i,2i+1}(u,v)=(J_{2i}u,J_{2i+1}v)$ satisfy   
    \[\|J_{2i}u+J_{2i+1}v\|_X\le (\|u\|^2_2+\|v\|^2_2)^{1/2}\le
    (1+\gamma_{i+1})\|J_{2i}u+J_{2i+1}v\|_X.
    \]
\end{enumerate}

\end{lemma}

To proceed without interruption, we demonstrate how
Lemma
\ref{L:3rdFDD} is applied to derive Theorem \ref{T:Main}, while
its proof is postponed to the end of the section.

At this point, a low-distortion embedding 
 $\Phi:\mM\to X$   will be
constructed as a piecewise defined map.

Let $R_0=0$. For any two nonnegative integers $j,k$ $(j<k)$,
consider the annulus
\[\ma_{j,k}=\{m\in\mM:~ R_j\le d(m,0)\le R_{k}\}.\]
Obviously when $j=0$ it is a ball. 
Observe that
$\{F_i\}_{i=1}^\infty$ forms an increasing sequence of subspaces of
$\ell_2$. Consequently there exist natural isometric embeddings of
$F_i$ into $F_{i+1}$.

First, we define a sequence of embeddings of annuli
$\ma_{2i,2i+3}$ into sums of the form $F_i\oplus_{Z_i} F_i$ and
$F_i\oplus_2 F_{i+1}$ as restrictions of bendings according to the
following procedure:

\begin{itemize}
    \item Define $\mt_1:\ma_{0,3}\to F_1\oplus_{Z_1}F_1$ as the restriction 
    to $\ma_{0,3}$ of the existing by Theorem \ref{T:BendDetail} $\left(\frac{1+\psi}{1-\psi}\right)$-bending
    of $F_1$ into $F_1\oplus_{Z_1}F_1$ with parameters $(R_1,R_2)$.

     \item Consider the  restriction  to $\ma_{2,5}$ of the existing by Theorem \ref{T:BendDetail} 
     $\left(\frac{1+\psi}{1-\psi}\right)$-bending of $F_2$ into $F_2\oplus_2 F_2$
     with parameters $(R_3,R_4)$. 
    Observe that because $\ma_{2,4}$ is a subset of $F_1$,
     the formula \eqref{E:DefBend} for bending implies that the image of this map is contained in $F_1\oplus_2 F_2$.
     Define $\mt_2:\ma_{2,5}\to F_1\oplus_2 F_2$ as the
     resulting map.
     
     \item \dots

     \item  Define $\mt_{2i-1}:\ma_{4i-4,4i-1}\to F_i\oplus_{Z_i}F_i$ as the restriction to $\ma_{4i-4,4i-1}$ of the existing by Theorem \ref{T:BendDetail} $\left(\frac{1+\psi}{1-\psi}\right)$-bending
    of $F_i$ into $F_i\oplus_{Z_i}F_i$ with parameters $(R_{4i-3},R_{4i-2})$.

     \item  Consider the  restriction  to $\ma_{4i-2,4i+1}$ of the existing by Theorem \ref{T:BendDetail} 
     $\left(\frac{1+\psi}{1-\psi}\right)$-bending of $F_{i+1}$ into $F_{i+1}\oplus_2 F_{i+1}$
     with parameters $(R_{4i-1},R_{4i})$. 
    Observe that because $\ma_{4i-2,4i}$ is a subset of $F_i$,
     the formula \eqref{E:DefBend} for bending implies that the image of this map is contained in $F_i\oplus_2 F_{i+1}$.
     Define $\mt_{2i}:\ma_{4i-2,4i+1}\to F_i\oplus_2 F_{i+1}$  as the
     resulting map.
     
     \item \dots
\end{itemize}

 To get embeddings into $V\subset X$, we consider compositions:

\[\Phi_{2i-1}:=\mj_{2i-1,2i}\circ\mt_{2i-1}:\ma_{4i-4,4i-1}\to V_{2i-1}\oplus V_{2i}
\]
and
\[\Phi_{2i}:=\mj_{2i,2i+1}\circ\mt_{2i}:\ma_{4i-2,4i+1}\to V_{2i}\oplus V_{2i+1}.
\]

Our next goal is to show that combining these maps we get a
well-defined $(1+\ep)$-bilipschitz map of $\mM$ into $V\subset X$.

We start with checking that on $\ma_{4i-2,4i-1}$, where both
$\Phi_{2i-1}$ and $\Phi_{2i}$ are defined, they coincide. Similarly,
we need to check that on  $\ma_{4i-4,4i-3}$ where both $\Phi_{2i-2}$ and $\Phi_{2i-1}$ are defined, they coincide. The proofs
are the same. We do it only for the first case. The maps
$\mt_{2i-1}$ and $\mt_{2i}$ map $\ma_{4i-2,4i-1}$ isometrically
into $F_i$. Since both $\mj_{2i-1,2i}$ and $\mj_{2i,2i+1}$ map
$F_i$ onto $V_{2i}$ using $J_{2i}$, the maps coincide.

Therefore, the formula 
\[ \Phi (x) = 
\begin{cases}
\mj_{2i-1,2i}\circ\mt_{2i-1} (x) \in V_{2i-1}\oplus V_{2i} &  \hbox{ if } x\in \ma_{4i-4,4i-1},  \ \ \ i\in \mathbb{N}, \\
\mj_{2i,2i+1}\circ\mt_{2i}(x) \in V_{2i}\oplus V_{2i+1} &  \hbox{ if } x\in \ma_{4i-2,4i+1}, \ \ \ i\in \mathbb{N},
\end{cases}
\]
in which we consider $V_{2i-1}\oplus V_{2i}$ as subspaces of $V$, gives a well-defined map. 
It remains
to prove that $\Phi$ is a $(1+\ep)$-bilipschitz embedding. 
To achieve this goal it suffices to establish bilipschitz inequalities in the three cases:

{\bf Case 1.} $x,y\in\ma_{4i-4,4i-1}$. Since
$\Phi_{2i-1}=\mj_{2i-1,2i}\circ\mt_{2i-1}$, by Theorem
\ref{T:BendDetail},
\[
(1-\psi)\|x-y\|_2\le\|\mt_{2i-1}x-\mt_{2i-1}y\|_{F_i\oplus_{Z_i}
F_i}\le (1+\psi)\|x-y\|_2,
\]
 and by Lemma~\ref{L:3rdFDD} {\bf (b)}, 
\begin{equation}\label{E:Case1}
\frac{1-\psi}{(1+\zeta)(1+\gamma)^2}\|x-y\|_2\le\|\Phi_{2i-1}x-\Phi_{2i-1}y\|_X\le
(1+\psi)\|x-y\|_2.
\end{equation}

{\bf Case 2.} $x,y\in\ma_{4i-2,4i+1}$. Since
$\Phi_{2i}=\mj_{2i,2i+1}\circ\mt_{2i}$, by Theorem
\ref{T:BendDetail},
\[
(1-\psi)\|x-y\|_2\le\|\mt_{2i}x-\mt_{2i}y\|_{F_i\oplus_2
F_{i+1}}\le (1+\psi)\|x-y\|_2,
\]
 and by Lemma~\ref{L:3rdFDD} {\bf (c)}, 
\begin{equation}\label{E:Case2}
\frac{1-\psi}{1+\gamma}\|x-y\|_2\le\|\Phi_{2i}x-\Phi_{2i}y\|_X\le
(1+\psi)\|x-y\|_2.
\end{equation}

{\bf Case 3.} $x$ and $y$ are not in the same annulus of the form
$\ma_{2i,2i+3}$. Obviously, it suffices to consider the case
$\|y\|\le \|x\|$.
Let $R_{2i}$ be the smallest ``even'' $R$ such that
$\|y\|  \le R_{2i}$.
Then necessarily $R_{2i+1} \le
\|x\|$, for otherwise $x$  and $y$ would both be in  $\ma_{2i-2,2i+1}$.
Applying condition \eqref{I:3} for choosing
$R_{2i+1}$, one obtains that in
this case $\|y\|\le\frac{\ep}d\|x\|$, and
\begin{equation}\label{E:EpEst1}\left(1-\frac{\ep}d\right)\|x\|\le
\|x\|-\|y\|\le
\|x-y\|\le\|x\|+\|y\|\le\left(1+\frac{\ep}d\right)\|x\|.
\end{equation}
We recall the fact that $\mt_j$ are norm-preserving. Together with
inequalities for $\mj_{j,j+1}$ in Lemma \ref{L:3rdFDD}, it implies
\begin{equation}\label{E:EpEst}\begin{split} \left(\frac1{(1+\zeta)(1+\gamma)^2}-\frac{\ep}{d}\right)\|x\|&\le
\frac1{(1+\zeta)(1+\gamma)^2}\|x\|-\|y\|\le \|\Phi x\|-\|\Phi y\|\\
&\le \|\Phi x-\Phi y\|\le\|\Phi x\|+\|\Phi y\|\\&
\le
\|x\|+\|y\|
\le
\lp{1+\frac{\ep}d}\|x\|.\end{split}\end{equation}
Combining \eqref{E:EpEst1} and \eqref{E:EpEst}, we get
\begin{equation}\label{E:EpEstFin}\begin{split}
\frac1{1+\frac{\ep}d}\left(\frac1{(1+\zeta)(1+\gamma)^2}-\frac{\ep}{d}\right)\,\|x-y\|&\le
\|\Phi x-\Phi y\|\\&\le
\frac{1+\frac{\ep}d}{1-\frac{\ep}d}\,\|x-y\|.
\end{split}
\end{equation}
The conclusion that $\Phi$ is a $(1+\ep)$-bilipschitz embedding of
$\mM$ into $X$ now follows from the choice of $\gamma,\psi,\zeta$,
and $d$ made at the beginning of Section \ref{S:ConstrEmb}.
\end{proof}

To complete the picture, we need to prove Lemma \ref{L:3rdFDD}.
This is done in the remaining part of Section \ref{S:ConstrEmb}.

\begin{proof}[Proof of Lemma \ref{L:3rdFDD}] Let $\zeta,\gamma,\{\gamma_i\}_{i=1}^\infty\in(0,1)$ be the
numbers, and $\{F_i\}_{i=1}^\infty$ be the subspaces of $\ell_2$
introduced at the beginning of Section \ref{S:ConstrEmb}, and, as
before, $n_i=\dim F_i$, $i\in \mathbb{N}$.

Applying Theorem \ref{T:AUC} with $\ep=\zeta$ and $A=1$, for each 
$n_i$ there exists $N_i\in \mathbb{N}$
such that any direct sum $\ell_2^{N_i}\oplus \ell_2^{N_i}$
with direct sum projections of norm $1$ contains a 
$\zeta$-invariant sub-sum $\ell_2^{n_i}\oplus
\ell_2^{n_i}$.\medskip 

In what follows, our construction of FDD uses the Mazur method for
constructing basic sequences \cite[p.~4]{LT77}. To implement it,
we need the definition below.

\begin{definition}
\label{D:l_norming} Let $\Omega\in(0,1]$. A subspace $\mN\subset
X^*$  is called {\it $\Omega$-norming over a subspace $Y\subset
X$} if
\[\forall y\in Y~\sup\{f^*(y) \ :~ f^*\in \mN,~\|f^*\|\le 1\}\ge\Omega\|y\|.\]
\end{definition}

 Let $N\in\mathbb{N}$ and let $\eta>0$. Denote by
$K(N,\eta)\in\mathbb{N}$ the least number for which the unit
sphere of any $N$-dimensional normed space contains a $\eta$-net
of cardinality at most $K(N,\eta)$. It is well known that such
$K(N,\eta)$ exists (see, for example, \cite[Lemma 9.18]{Ost13}).

Since $X$ is infinite-dimensional, by the Dvoretzky Theorem,
 there is a subspace
$U_1\subset X$ with $\dim U_1=N_1+K(N_1,\gamma_1/(1+\gamma_1))$
and $d_{\rm BM}(U_1,\ell_2^{\dim U_1})\le (1+\gamma_1)$. We pick a
finite-dimensional subspace $\mN_1\subset X^*$ which is
$\frac1{1+\gamma_1}$-norming over $U_1$ (see, for example,
\cite[Lemma 4.2]{OO19} for the proof of existence of such
subspace).

Using the Dvoretzky Theorem again, 
we find a subspace $U_2\subset
(\mN_1)_\top:=\{x\in X:~ x^*(x)=0~\forall x^*\in\mathcal{N}_1\}$
such that
\[\dim U_2=N_1+N_2+K(N_2,\gamma_2/(1+\gamma_2))\]
and $d_{\rm BM}(U_2,\ell_2^{\dim U_2})\le (1+\gamma_2)$. Next, we pick a  
finite-dimensional   subspace $\mN_2\subset X^*$ which is
$\frac1{1+\gamma_2}$-norming over $\lin(U_1\cup U_2)$. Proceeding
like this, in Step $k$ we apply the Dvoretzky  Theorem to find a
subspace $U_k\subset (\mN_{k-1})_\top$ such that
\[\dim U_k=N_{k-1}+N_k+K(N_k,\gamma_k/(1+\gamma_k))\]
and $d_{\rm BM}(U_k,\ell_2^{\dim U_k})\le (1+\gamma_k)$. Next, we pick a 
 finite-dimensional subspace $\mN_k\subset X^*$ which is
$\frac1{1+\gamma_k}$-norming over
$\lin\left(\bigcup_{i=1}^kU_i\right)$, and so on.

The fact that the sequence $\{U_i\}_{i=1}^\infty$ forms an FDD of
its closed linear span can be derived from the following lemma.

\begin{lemma}\label{L:ProjKer}
Let $\mf$ be a subspace in $X$. If a subspace
$\mN \subset X^*$ is $\Omega$-norming over $\mf$, and $\me$ is a
subspace of $\mN_\top$, then the projection $P:\me\oplus\mf\to\mf$
given by $P(e+f)=f$, where $e\in\me$, $f\in\mf$ satisfies
$\|P\|\le1/\Omega$.
\end{lemma}
\begin{proof}
We need to show that $\|f\|\le \frac1\Omega\|e+f\|$. 
Let $\ve>0$ be arbitrary and  $f^*\in
\mN$ be such that $\|f^*\|=1$ and $f^*(f)\ge \lp{\Omega -\ve}\|f\| $. Since
$e\in\mN_\top$, one has \[\|e+f\|\ge f^*(e+f)\ge
\lp{\Omega -\ve}\|f\|.\]
Since $\ve>0$ is arbitrary, the proof is completed.
\end{proof}

Lemma \ref{L:ProjKer} with $\mf= \lin\{U_i\}_{i=1}^k$  and $\me=\lin\{U_i\}_{i=k+1}^\infty$ 
implies that for any $k\in\mathbb{N}$ the projection $P_k$
of $\lin\{U_i\}_{i=1}^\infty$ onto
$\mf$ containing $\me$ in the kernel has norm $\le(1+\gamma_k)$. 
Now, the
standard argument  \cite[p.~47]{LT77} implies that
$\{U_i\}_{i=1}^\infty$ 
 forms an FDD of its linear span.
 It also follows that for all $i\ge 1$ the projections $U_i\oplus U_{i+1} \to U_{i}$ given by $(x_1,x_2) \to x_1$,
  have norm $\le(1+\gamma_i)$. 

Further, we define subspaces $\{W_{i}\}_{i=1}^\infty$ as follows.
The subspace $W_{2i}$, $i\in\mathbb{N}$, is picked as an arbitrary
$N_{i}$-dimensional subspace of $U_{i+1}$.

Before defining $\{W_{2k-1}\}_{k=1}^\infty$, we endow each $U_i$
with a Euclidean inner product and norm from a Euclidean space
$\tilde U_i$ on which the Banach-Mazur distance
$d_{\rm BM}(U_i,\ell_2^{\dim U_i})$ is attained. 

We denote this norm on $U_i$ by $\nm{\cdot}_\sim^i$ and assume that it 
 satisfies the condition
\begin{equation}
\label{E:quantBM}
\forall x\in U_i\quad
\frac1{1+\gamma_i}\|x\|^i_{\sim} \le \|x\|_X \le   \|x\|^i_{\sim}.
\end{equation}

Let $G_{2i+1}$ be the orthogonal complement of $W_{2i}$ in $U_{i+1}$ 
 endowed with the inner product of $\tilde
U_{i+1}$, and $G_1$ be $U_1$. As such, $G_i$ is defined for odd
$i$ only.

We say that a set $\mathcal{D}$ is {\it $\eta$-dense} $(\eta>0)$
in a metric space $\mathcal{M}$ if, for every $m\in\mathcal{M}$,
there is $x\in\mathcal{D}$ such that $\|m-x\|<\eta$.

By the definition of $K(N,\eta)$, there is a
$\gamma_i/(1+\gamma_i)$-dense set $\mathcal{D}_i$ of cardinality
$K(N_i,\gamma_i/(1+\gamma_i))$ in the unit sphere $S(W_{2i})$. For
each $w\in\mathcal{D}_i$, consider a supporting functional
$w^*_w\in X^*$ such that $w^*_w(w)=\|w^*_w\|=1$. The choice of
dimension of $G_{2i-1}$ is such that the intersection
\begin{equation}\label{E:Intersect}
G_{2i-1}\bigcap(\cap_{w\in\mathcal{D}_i}\ker w^*_w)
\end{equation} 
has dimension at least $N_i$ (it can be more because some 
of the supporting functionals can be linearly dependent). 
We pick in the intersection \eqref{E:Intersect} a subspace 
of dimension $N_i$ and denote it $W_{2i-1}$.

The verification that the functionals
$\{w^*_w\}_{w\in\mathcal{D}_i}$ span a subspace which is
$\left(1-\frac{\gamma_i}{1+\gamma_i}\right)=\frac1{1+\gamma_i}$-norming
over $W_{2i}$ is immediate. 
Applying Lemma \ref{L:ProjKer} again, 
with $\mf= W_{2i}$  and $\me=W_{2i-1}$, we obtain that
 the projections $W_{2i-1}\oplus W_{2i}\to W_{2i}$ given by $(x_1,x_2) \to x_2$,
  have norm $\le(1+\gamma_i)$. 
Therefore,  we conclude that the norms of both direct sum projections in the
direct sum $W_{2i-1}\oplus W_{2i}$ $(i\in \mathbb{N})$ do not
exceed $(1+\gamma_i)$; recall that the bound on the direct sum projection 
$W_{2i-1}\oplus W_{2i}\to W_{2i-1}$ follows  from the bound on the projection $U_i\oplus U_{i+1} \to U_{i}$.

The fact that $\{W_i\}_{i=1}^\infty$  
 forms an FDD in its closed
linear span follows from the criterion in \cite[p.~47]{LT77}.

Finally, we prove the next auxiliary result.

\begin{lemma}\label{L:Ren_2i-1_2i}  Let $\nm{ \cdot}^i_N$ be the following norm on 
$W_{2i-1}\oplus W_{2i}$:   
\[\|(x_1,x_2)\|^{i}_N=\max\{\|x_1\|^{i}_{\sim}, \|x_2\|^{i+1}_{\sim}, 
 \ \|(x_1,x_2)\|\}, \]
 where $\|(x_1,x_2)\|$ means $\|x_1+x_2\|_X$ and $\nm{\cdot}_\sim^i$ is the introduced above norm on $U_i$. 
Then, the space $\lp{W_{2i-1}\oplus W_{2i}, \nm{ \cdot}^i_N}$ has 
the direct sum projections of norm $1$ and 
\[ \|(x_1,x_2)\|  \le\|(x_1,x_2)\|^i_N  \le (1+\gamma_i)^2\|(x_1,x_2)\|. \]
\end{lemma}
\begin{proof} 
The statement about the norms of projections is immediate from the
definition. 

 Let $(x_1,x_2)\in W_{2i-1}\oplus W_{2i}$.
Recall that
from the norms of the direct sum projections (as linear maps between subspaces of $X$ with the induced norm) we have
\[\|x_1\|_X \le (1+\gamma_{i})\|(x_1,x_2)\|,\qquad \|x_2\|_X \le
(1+\gamma_{i})\|(x_1,x_2)\|.\]
From the construction of  $\{U_i\}_{i=1}^\infty$ we have 

\[ \frac1{1+\gamma_i}\|x_1\|^i_{\sim} \le \|x_1\|_X \le \|x_1\|^i_{\sim}, \]
and 
\[ \frac1{1+\gamma_i} \|x_2\|^{i+1}_{\sim}  \ \  
 \le\frac1{1+\gamma_{i+1}}\|x_2\|^{i+1}_{\sim} \ \ 
 \le \|x_2\| _X \ \ \le \|x_2\|^{i+1}_{\sim} .\]

Therefore,
\[\begin{split}
\|(x_1,x_2)\|&\le\|(x_1,x_2)\|^{i}_N \\&
=\max\{\|x_1\|^{i}_{\sim}, \|x_2\|^{i+1}_{\sim}\ , \|(x_1,x_2)\|\}
\\& \le \max\{(1+\gamma_i)\|x_1\|_X,
(1+\gamma_i)\|x_2\|_X, \|(x_1,x_2)\|\}.
\end{split}
\]

Hence, 
\[ \|(x_1,x_2)\|  \le\|(x_1,x_2)\|^{i}_N  \le (1+\gamma_i)^2\|(x_1,x_2)\|. \qedhere\]
\end{proof}

Next, we apply Theorem \ref{T:AUC} to each of the sums
$(W_{2i-1}\oplus W_{2i}, \|\cdot\|^i_N)$ for all $i\in \mathbb{N}$. 
We find subspaces of dimension $n_i$, which we denote $F_{i}' \subset W_{2i-1}$ and $F_{i}'' \subset W_{2i}$ such that 
the norm $\nm{\cdot}^i_N$ restricted to 
$F_i'\oplus F_i''$ is $\zeta$-invariant. 

Applying Lemma~\ref{L:InvNorm}  we obtain that the norm 
\[|||y_1+y_2|||^i :=\sup_{O_1,O_2{\rm~orthogonal~on~}
 F_i', F_i''}\| \lp{O_1y_1,O_2y_2}\|^i_N,\quad y_1\in F_i', y_2\in  F_i''\]
 on 
$F_i'\oplus F_i''$ satisfies
\begin{equation}
\label{E:toCond}
\|\lp{y_1,y_2}\|^i_N\le |||y_1+y_2|||^i\le (1+\zeta)\|\lp{y_1,y_2}\|^i_N
\end{equation}
 and 
\[ |||O_1y_1+O_2y_2|||^i=|||y_1+y_2|||^i \]
for every orthogonal operators $O_1$ on $F_i'$ and $O_2$ on $F_i''$. 

By Lemma~\ref{L:InvToUncond}  there exists a $1$-unconditional norm
$\|\cdot\|_{Z_i}$ on $\mathbb{R}^2$ such that for any $ y_1\in F_i'$ and  $y_2\in  F_i''$
\[ 
|||y_1+y_2|||^i=\|(\|y_1\|^i_\sim,\|y_2\|^{i+1}_\sim)\|_{Z_i}.
\]

For $i\in \mathbb{N}$ define  $V_{2i-1}$ as the subspace of $X$ that coincides with $F'_i$   and $V_{2i}$ 
as the subspace of $X$ that coincides with $F''_i$. 

We choose isometries $I_i':F_i\to F_i'$ and $I_i'':F_i\to F_i''$ 
and define $J_{2i-1}:F_i\to V_{2i-1}$ and $J_{2i}:F_i\to V_{2i}$ 
as compositions of these isometries and the mentioned above 
natural maps of $F_i'$ onto $V_{2i-1}$ and $F_i''$ onto $V_{2i}$. 

For $v\in F_i \subset \ell_2$ we have 
\[ \nm{v}_2 = \nm{J_{2i-1}v}_\sim^{i}\le (1+\gamma_{i})\nm{J_{2i-1}v}_X
\le   (1+\gamma_{i})\nm{J_{2i-1}v}_\sim^{i}=  (1+\gamma_{i})\nm{v}_2,\]
and 
\[ \nm{v}_2 = \nm{J_{2i}v}_\sim^{i+1}\le (1+\gamma_{i+1})\nm{J_{2i}v}_X
\le   (1+\gamma_{i+1})\nm{J_{2i}v}_\sim^{i+1}=  (1+\gamma_{i+1})\nm{v}_2.\]
This is property {\bf (a)} in the Lemma~\ref{L:3rdFDD}.

To prove property {\bf (b)}, note that for $(u,v) \in F_i'\oplus F_i''$, by \eqref{E:toCond} we have
\[  \nm{ \lp{J_{2i-1}u, J_{2i}v}}^i_N \le  ||| J_{2i-1}u+ J_{2i} v |||^i \le (1+\zeta) \nm{ \lp{J_{2i-1}u, J_{2i}v}}^i_N. \] 
Using the inequality on the left and Lemma~\ref{L:Ren_2i-1_2i}   , we get
\[ \begin{aligned}
& \nm{ J_{2i-1}u+ J_{2i}v}_X \le \nm{ \lp{J_{2i-1}u, J_{2i}v}}^i_N \le   ||| J_{2i-1}u+ J_{2i} v |||^i  \\ 
& =  \nm{\lp{\nm{J_{2i-1}u}^i_\sim, \nm{J_{2i}v}^{i+1}_\sim }}_{Z_i} 
= \nm{\lp{\nm{u}_2, \nm{v}_2}}_{Z_i},
\end{aligned}\]
which is the inequality on the left in {\bf (b)}. 

On the other hand, we have 
\[ \begin{aligned}
& \nm{\lp{\nm{u}_2, \nm{v}_2}}_{Z_i}  =  \nm{\lp{\nm{J_{2i-1}u}^i_\sim, \nm{J_{2i}v}^{i+1}_\sim }}_{Z_i}   = ||| J_{2i-1}u+ J_{2i} v |||^i \\
& \le (1+\zeta) \nm{ \lp{J_{2i-1}u, J_{2i}v}}^i_N  \le (1+\zeta)(1+\gamma_i)^2 \nm{ J_{2i-1}u + J_{2i}v}_X, 
\end{aligned}\]
the last inequality being a consequence of Lemma~\ref{L:Ren_2i-1_2i}.

To prove property {\bf (c)} note that because $V_{2i}$ and $V_{2i+1}$ are orthogonal subspaces of the Euclidean space $\tilde{U}_{i+1}$ 
(i.e. the space $U_{i+1}$ endowed with the norm $\nm{\cdot}^{i+1}_\sim$), for $u\in F_i$ and $v\in F_{i+1}$ 
we have 
\[ \begin{aligned}
& \nm{ J_{2i}u + J_{2i+1}v}_X \stackrel{\eqref{E:quantBM}}{\le}  \nm{ J_{2i}u + J_{2i+1}v}^{i+1}_\sim 
= \lp{ \lp{ \nm{ J_{2i}u }^{i+1}_\sim}^2 + \lp{ \nm{ J_{2i+1}v }^{i+1}_\sim}^2}^\frac12 \\
& =  \lp{\nm{u}^2_2+ \nm{v}^2_2}^\frac12.
\end{aligned}\]
On the other hand, 
\[ \begin{aligned}
& \lp{\nm{u}^2_2+ \nm{v}^2_2}^\frac12 
 = \lp{ \lp{ \nm{ J_{2i}u }^{i+1}_\sim}^2 + \lp{ \nm{ J_{2i+1}v }^{i+1}_\sim}^2}^\frac12 = 
\nm{ J_{2i}u + J_{2i+1}v}^{i+1}_\sim \\
& \le (1+\gamma_{i+1})\nm{ J_{2i}u + J_{2i+1}v}_X.
\end{aligned}\]
Therefore  property {\bf (c)} also holds. 

In conclusion, $\{V_i\}_{i=1}^\infty$ forms an FDD of its closed linear span, satisfying all conditions of Lemma
\ref{L:3rdFDD}.
\end{proof}

\section{A counterexample to the general bending
problem}

It would be very interesting to prove analogues of our main
result, Theorem \ref{T:Main}, for spaces which are different from
the Hilbert space. To state some relevant problems, we recall
that a Banach space $W$ is said to be {\it finitely represented}
in a Banach space $X$ if,  for every $\ep>0$ and every
finite-dimensional subspace $F$ in $W$, there is a
finite-dimensional subspace $G$ in $X$ such that $\dim G=\dim F$
and $d_{\rm BM}(F,G)\le 1+\ep$.
\medskip

 The first question of interest is the following

\begin{problem}\label{P:FinRepLF} Let $\mM$ be a locally finite
subset of an infinite-dimensional Banach space $W$ and assume that
$W$ is finitely represented in a Banach space $X$. Does it imply
that, for every $\ep>0$, the space $\mM$ admits a
$(1+\ep)$-bilipschitz embedding into $X$?
\end{problem}

To pave a way towards solving this problem, it is desirable to obtain an affirmative answer to the problem below. Notice that its formulation uses
Definition \ref{D:bend}.
\medskip

\noindent{\bf General Bending Problem:} {\it Let $X$ and $Y$ be
finite-dimensional Banach spaces such that there exist two linear
isometric embeddings $I_1:Y\to X$ and $I_2:Y\to X$ with distinct
images, $Y_1=I_1(Y)$ and $Y_2=I_2(Y)$. Assume that $X$ is the
direct sum of $Y_1$ and $Y_2$ and that the direct sum projections
of $X=Y_1\oplus Y_2$ have norm $1$. Does it imply that for every
$\ep>0$ there exist $(r,R)$ with $0<r<R<\infty$ for which there
exists a $(1+\ep)$-bending of $Y$ in the space $X$ from $I_1$ to
$I_2$ with parameters $(r,R)$?}
\medskip

However, as the following theorem shows, the answer to this problem is negative even in the case where $Y$ is
a two-dimensional Euclidean space. Thence, the General Bending Problem
as stated above is excessively strong, one should look for weaker
statements which might be true. Also, perhaps
suitable developments of Theorem \ref{T:CountGenBendP} can be used  to
obtain the affirmative answer to  the question of Problem \ref{P:OO19}.
\medskip

\begin{theorem}\label{T:CountGenBendP} There exists a
$4$-dimensional Banach space $X$ satisfying the conditions:
\begin{enumerate}[{\bf (A)}]

\item It is a direct sum of two $2$-dimensional Euclidean spaces
$Y_1$ and $Y_2$ with direct sum projections having norm $1$.

\item There exists $\ep>0$ such that for any $(r,R)$ satisfying
$0<r<R<\infty$ and any isometric embeddings $I_1:\ell_2^2\to Y_1$
and $I_2:\ell_2^2\to Y_2$, there is no $(1+\ep)$-bending with parameters $(r,R)$ of
$\ell_2^2$ in $X$ from $I_1$ to $I_2$.

\end{enumerate}
\end{theorem}

Recall that, for a Banach space $X$,  $S(X)$ denotes the unit sphere in $X$.
The {\it spherical
opening} between subspaces $U$ and $W$ of a Banach space $X$ is defined as:
\[\Omega  (U,W)=\max \{\sup _{u\in S(U)}\hbox{dist}(u,S(W)),
\sup_{w\in S(W)}\hbox{dist}(w,S(U))\}.
\]
It is easy to see that $\Omega$ is a metric on the set of all closed subspaces of a Banach space, and that this metric space is compact if the Banach space is finite-dimensional.
We refer to \cite[Section 3.12]{Ost94} for more properties of this metric.
\medskip

\begin{lemma}\label{L:VerOrHor}
Let $Y_1$ and $Y_2$ be $2$-dimensional Euclidean spaces  and let $\delta >0$. 
There
exists a norm on $Y_1\oplus Y_2$ such that the obtained normed
space $(X,\|\cdot\|_X)$ satisfies the conditions:

\begin{enumerate}[{\rm (i)}]

\item \label{I:Isom} On each of the summands $Y_1$ and $Y_2$ the
norm is isometrically equivalent to its original norm - the
$\ell_2^2$ norm.

\item \label{I:Proj1} The projection onto any of the summands $Y_1$ or $Y_2$, whose kernel equals the other summand, has
norm $1$.

\item \label{I:NotEucl} For every  sufficiently small $\gamma>0$,  there exists $\ve(\gamma)>0$
 such that every two-dimensional subspace $Z$  of $X$ satisfying
$\Omega(Z,Y_1)\ge\gamma$ and $\Omega(Z,Y_2)\ge\gamma$, satisfies
$d_{\rm BM} (Z,\ell_2^2)\ge 1+\ve(\gamma)$, where $d_{\rm BM}$ is the
Banach-Mazur distance.

\item \label{I:CloseEucl} The norm of $X$ is not far from the norm
of $Y_1\oplus_2 Y_2$ denoted by  $\|\cdot\|_2$. Namely,
\begin{equation}\label{E:NormEquiv}
\forall x\in X \quad(1-\delta^2/2)\|x\|_X\le \|x\|_2\le \|x\|_X.
\end{equation}

\end{enumerate}

\end{lemma}

\begin{proof} The main idea of our proof of Lemma \ref{L:VerOrHor}
is to construct the unit ball of $X$ as the result of cutting from
the unit ball of the Euclidean space $Y_1\oplus_2 Y_2$ some
collection of symmetric pairs of caps.
By {\it cap} centered at a unit vector  $w$ in $\reo^4$ we mean the region of the unit ball in $\reo^4$ separated by a hyperplane orthogonal to the line spanned by $w$. The {\it radius} of the cap is the chordal (Euclidean) distance from $w$ to the $2$-dimensional sphere that is the intersection of the hyperplane and $S(\reo^4)$.
In our construction, these radii will be small enough to
satisfy  inequality \eqref{E:NormEquiv}.
In constructing the unit ball of $X$, sufficiently many caps will be removed so that
each two-dimensional subspace $G$ of $Y_1\oplus_2 Y_2$, except
$Y_1$ and $Y_2$, intersects the interior of at least one of the
caps  and, therefore, the norm of $X$ on $G$ will not be strictly convex;
consequently $G$ is not isometric to $\ell_2^2$.

It is clear that each space $X$ constructed as described above satisfies the conditions of
items \eqref{I:Isom}, \eqref{I:Proj1}, and \eqref{I:CloseEucl}.
\medskip

Now we prove that the condition in item \eqref{I:NotEucl} holds.
Let us assume the contrary. Then, for every $k\in \mathbb{N}$, there
exists $Z_k$ satisfying $\Omega(Z_k,Y_1)\ge\gamma$,
$\Omega(Z_k,Y_2)\ge \gamma$, and $d_{\rm BM}
(Z_k,\ell_2^2)<1+\frac1k$. Since the set of all subspaces of $X$
is compact with respect to the metric $\Omega$, the sequence
$\{Z_k\}_{k=1}^\infty$ has a $\Omega$-convergent subsequence. Let
$W$ be its limit.  The fact that for finite-dimensional spaces
$d_{\rm BM}$ is continuous with respect to $\Omega$ implies that $d_{\rm
BM}(W,\ell_2^2)=1$, and, thereupon, $W$ is isometric to $\ell_2^2$.

On the other hand, both $\Omega(W,Y_1)\ge\gamma$ and
$\Omega(W,Y_2)\ge\gamma$, whence $W$ is not the same as $Y_1$ or
$Y_2$, and hence its unit sphere contains line segments. This outcome contradicts  the conclusion of the previous paragraph.
\medskip

We now give details on how the removed caps are to be  selected.
Denote by $G_2(\mathbb{R}^4)$ the set of all two-dimensional
subspaces of $\mathbb{R}^4$. It is a compact space in the metric
$\Omega$. Let $\delta\in(0,\frac14)$. Using the standard approach,
we find in $G_2(\mathbb{R}^4)$  a finite subset $\Delta$ such that

\begin{enumerate}

\item $Y_1,Y_2 \in\Delta$,

\item $\forall W_1,W_2\in \Delta, ~~ W_1\ne W_2,
~~\Omega(W_1,W_2)\ge\delta$,

\item $\forall L\in G_2(\mathbb{R}^4), ~~ \exists W\in \Delta, ~~
\Omega(W,L)<\delta$.

\end{enumerate}

For each $W\in\Delta$ other than $Y_1$ or $Y_2$, we select a point $w\in S(W)$ which
is at distance at least $\delta$ to both $S(Y_1)$ and $S(Y_2)$.
We then cut from the unit ball of $\mathbb{R}^4$ two $0$-symmetric caps
of radius $\delta$ at $w$ and $-w$. It is clear that in such a
way we cut finitely many caps  and that ``under'' any of the caps
the resulting surface will be polyhedral.

Observe that the existence of $w$ is guaranteed for every
$W\in\Delta$ except $Y_1$ and $Y_2$. In fact, it is immediate that
there are $w_1$ and $w_2$ in $S(W)$ with both  $\dist(w_1,
S(Y_1))\ge\delta$ and $\dist(w_2,S(Y_2))\ge\delta$. If neither
$w_1$ nor $w_2$ works, meaning that both $\dist(w_1, S(Y_2))<\delta$
and $\dist(w_2,S(Y_1))<\delta$, then $\dist(w_1,S(Y_1))>\sqrt{2}-\frac14$ and
$\dist(w_2,S(Y_2))>\sqrt{2}-\frac14$. As a consequence, moving along the
sphere  $S(W)$ from $w_1$ to $w_2$ we arrive at the desired point.

We are ``almost'' done because,  for every $L\in G_2(\mathbb{R}^4)$,
there is $W\in \Delta$ such that $\Omega(W,L)<\delta$. If $W\ne
Y_1,Y_2$, we are done because the cap which we cut around the
point $w\in S(W)$ will cut some piece under $S(L)$. The only subspaces
$L$ which are not covered by this reasoning are those that  are in the set
\[\Psi:=\{L:~\min_{W\in \Delta, W\ne Y_1,Y_2}\Omega(L,W)\ge\delta\}.\]
This is a compact set. 
For this reason the function
\[\omega(L):=\min\{\Omega(L,Y_1),\Omega(L,Y_2)\}\]
attains its maximum on $\Psi$, and this maximum $\mu$ satisfies
$\mu<\delta$.

Consider an orthonormal basis $\lc{e_1, e_2, e_3, e_4}$ in $Y_1\oplus_2 Y_2=\reo^4$ such that 
$Y_1=\lin\lp{\lc{e_1,e_2}}$ and $Y_2=\lin\lp{\lc{e_3,e_4}}$. 
Choose $a>0$ in such a way that for the unit vector $f=\frac1{\sqrt{1+a^2}}e_1+ \frac{a}{\sqrt{1+a^2}}e_3$ we have 
$ \nm{e_1-f} = \delta$. Specifically, this condition means that $\frac1{\sqrt{1+a^2}}=1-\frac{\delta^2}{2}$. 

Let $\sigma= \frac1{\sqrt{1+a^2}}$ and $\tau = \frac{a}{\sqrt{1+a^2}}$.
We remove 16 caps  of radius $\delta$, tangent to $S(Y_1)$, centered at the points with position vectors 
$\lp{ \pm \sigma e_1\pm \tau e_3}$, $\lp{ \pm \sigma e_1\pm \tau e_4}$, and 
$\lp{ \pm \sigma e_2\pm \tau e_3}$, $\lp{ \pm \sigma e_2\pm \tau e_4}$. 
Similarly, we remove the 16 caps  of radius $\delta$, tangent to $S(Y_2)$, 
centered at the points with position vectors 
$\lp{ \pm \sigma e_3\pm \tau e_1}$, 
$\lp{ \pm \sigma e_3\pm \tau e_2}$, 
and  
$\lp{ \pm \sigma e_4\pm \tau e_1}$, 
$\lp{ \pm \sigma e_4\pm \tau e_2}$.

We now prove that for each $L\in\Psi$ there will be some part cut out of $S(L)$ by some of the caps described above.  

Let us choose $L\in\Psi$ and since one (and only one) of the conditions
$\Omega(L,Y_1)<\delta$ or $\Omega(L,Y_2)<\delta$ holds we can assume that $\Omega(L,Y_1)<\delta$. 
First, we argue that $S(L)$ intersects the hyperplane $\lin\lp{\lc{e_1,e_3,e_4}}$ at unique point of position vector $l$ so that 
$\nm{e_1-l} <\delta$. Note that $L$ cannot be a subspace of $\lin\lp{\lc{e_1,e_3,e_4}}$ since that would imply  $\Omega(L,Y_1) =\sqrt{2}$. 

Since $S(L)$ is symmetric about the origin, if $(x_1,x_2,x_3,x_4)\in S(L)$, then so is its opposite, 
and because the coordinate functions are continuous we necessarily have two diametrically opposite points with the coordinate 
$x_2=0$ (there are only two such points, for otherwise   $L \subset \lin\lp{\lc{e_1,e_3,e_4}}$). 
Let $\pm l$ be the position vectors of the two points $\pm (x_1,0,x_3,x_4)\in S(L)$. 
Since   $\Omega(L,Y_1) <\delta$ we have that $\dist(l, S(Y_1))< \delta$ and therefore 
$ \min_t \lc{ (x_1-\cos t)^2+ (0-\sin t)^2 + x_3^2+x_4^2} < \delta^2$, i.e. 
$\min_t \lc{2-2x_1\cos t} <\delta^2$.
 Note that $\dist(l, S(Y_1))$ is achieved when $x_1\cos t = \av{x_1}$ and without loss of generality we will assume that 
 $x_1>0$ and therefore $t=0$, i.e. the vector on $S(Y_1)$ closest to $l$ is $e_1$. 
Moreover, we may assume without loss of generality that 
$ l = \frac{1}{\sqrt{1+b^2+c^2}}\lp{e_1+be_3+ce_4}$ for coefficients $ b\ge c \ge 0$ where at least $b$ is positive.  
Indeed, if $b=c=0$, then $l = e_1$ and in this case we repeat the argument near the vector $e_2$ where we search for points
in $S(L) \cap \lin\lp{\lc{e_2,e_3,e_4}}$. Again, this intersection consists of a vector and its opposite. 
This time the vector near $e_2$ cannot coincide with $e_2$ for this would imply $L=Y_1$. 
If this happens, then we swap the labels of $e_1$ and $e_2$ and we are in the situation claimed above, 
with $l\neq e_1$ and $b> 0$. 

To show that a nonempty part will be cut out of $S(L)$, we show that $l$ is in the open cap of radius $\delta$ centered at 
 $f=\sigma e_1+ \tau e_3$.
 For this it suffices to show the inequality
 $\inn{f}{l} > \inn{f}{e_1}$ 
between inner products of unit vectors.  
It is equivalent to 
\begin{equation}
\label{E:ineq}
 \frac{1+ab}{\sqrt{1+b^2+c^2}}>1.
 \end{equation}
We remark that $\nm{e_1-f}=\delta > \nm{e_1-l}$ is equivalent to 
 $\inn{e_1}{l} > \inn{e_1}{f}$ which means  
 \[ \frac{1}{\sqrt{1+b^2+c^2}} > \frac{1}{\sqrt{1+a^2}}, \] 
 and therefore $a >  b$. 
 
 We thus have 
 \[ (1+ab)^2 > 1+2ab > 1+2b^2 \ge 1+b^2+c^2, \]
 which implies \eqref{E:ineq}.

Deleting these $32$ caps  together with caps centered at $w\in S(W)$ chosen above
from the unit ball of
$\mathbb{R}^4$, we get the unit ball of $X$ satisfying all of the
conditions of Lemma \ref{L:VerOrHor}.
\end{proof}

\begin{proof}[Proof of Theorem \ref{T:CountGenBendP}] 
We are going to prove that there exists $\ve>0$ such that the space $X$ constructed in
Lemma~\ref{L:VerOrHor} does not admit a $(1+\ve)$-bending of
$Y=\ell_2^2$ with parameters $(r,R)$ for any $0<r<R<\infty$.

To prove the statement by contradiction, select
 \begin{equation}
\label{E:pick_gamma} \sqrt[6]{2}-1 > \gamma >0,
\end{equation}
so that
\[1> \frac{(1+\gamma)^3}{\sqrt{2}}. \]
Let $\ve(\gamma)$ be the value given by item (\ref{I:NotEucl}) in
Lemma~\ref{L:VerOrHor}.
 We pick $\ve >0$ so that
\begin{equation}
\label{E:pick_ep}
\ve < \min\lc{\gamma,\ve(\gamma)}.
\end{equation}
Finally, we choose $\delta>0$ such that
\begin{equation}
\label{E:pick_del}
1-\frac{\delta^2}{2}> \frac{(1+\gamma)^3}{\sqrt{2}}.
\end{equation}

Next, assume that there exists a $(1+\ep)$-bending $T:Y\to X$ with
parameters $(r,R)$, $0<r<R<\infty$. Conforming to  the notation above,
we write $T = (T_1,T_2)$ meaning
\[ T_1:Y\to Y_1 \ \ \mbox{ and } \ \ T_2:Y\to Y_2.\]

In view of the Rademacher theorem, this map is differentiable
almost everywhere. By a  standard argument,   the
derivative $DT(y)$, whenever it exists,  is a $(1+\ep)$-bilipschitz linear embedding of
$Y$ into $X$ (see \cite[Chapter 7, Section 1]{BL00}).

\begin{remark}
 \label{R:1or2}
Our construction of $X$ yields that, for $\ve
<\ve(\gamma)$, item \eqref{I:NotEucl} in Lemma \ref{L:VerOrHor}
implies that at every point of differentiability  $y\in Y$, either
\[\Omega(DT(y)Y,Y_1)<\gamma, \ \ \mbox{ or } \ \ \Omega(DT(y)Y,Y_2)<\gamma.\]
\end{remark}
Indeed, if both $\Omega(DT(y)Y,Y_1)$ and $\Omega(DT(y)Y,Y_2)$ are $\ge \gamma$, then 
 Lemma \ref{L:VerOrHor} item \eqref{I:NotEucl} implies that $d_{BM}(DT(y)Y,\ell_2^2) \ge 1+\ve(\gamma)$, 
 which contradicts the fact that $T$ is a $(1+\ve)$-bending of $Y=\ell^2_2$ with $\ve <\ve(\gamma)$.

Let us paint $Y$ in three colors:
\begin{itemize}
\item
blue for the points where  $DT(y)Y$ is close to $Y_1$,
\item
yellow for the points where $DT(y)Y$ is close to $Y_2$,
\item
red for the points where $DT(y)$ does not exist.
\end{itemize}
Note that since $\gamma $ is such that a two-dimensional subspace $Z$ of $X$ cannot have simultaneously 
$\Omega(Z,Y_1) <\gamma$ and $\Omega(Z,Y_2) <\gamma$, it follows that points of differentiability of $T$ cannot be simultaneously blue and yellow. 

We continue by proving the following statement.
There exists a line segment in $Y$ such that:
\begin{enumerate}
\item Almost all of its points are either blue or yellow.
\item The set of points which are blue takes half of its
measure.
\end{enumerate}

To prove this statement
consider the $0$-centered disc
of radius $r$ in $Y$.
We fix Cartesian coordinates $(x,y)$ in $Y$ and denote by $u$ the unit vector in the positive $y$-direction.
Consider the set of  all vertical (parallel to $u$)  $x$-axis-symmetric
line segments $I_{x}$
of length $2R+r$,  whose intersection with the
disc are of length at least $r$ (see Figure~\ref{F:equi}).
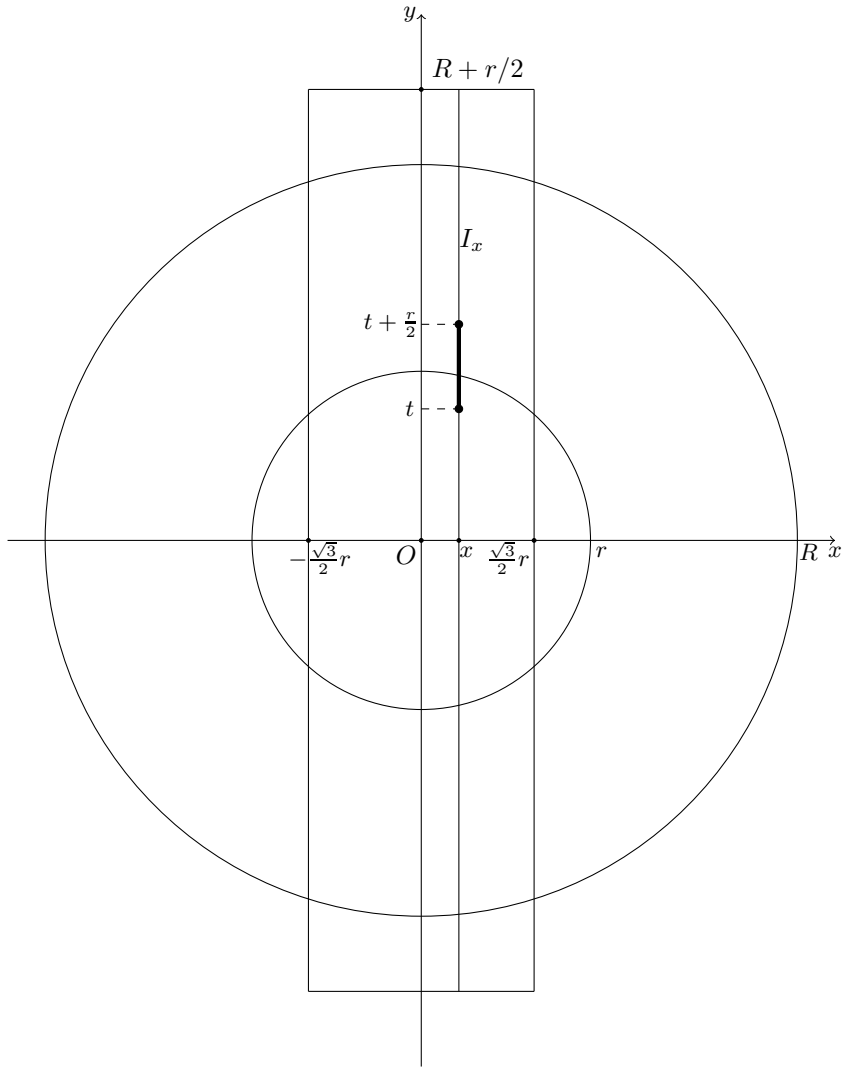
\begin{figure}
\begin{tikzpicture}[scale=0.5]
\coordinate (O) at (0,0);

\draw[black,thin] (O) circle(10);
\node[black,thin] at (10.3,-0.3){\small $R$};
\draw[black, thin] (O) circle (4.5);
\node[black] at (4.8,-0.3) {\small $r$};

\filldraw (O) circle (0.05);

\coordinate (X) at (-11,0);
\coordinate (XX) at (11,0);

\draw[thin,->] (X)--(XX);
\node at (11,-0.3)  {\small $x$ };

\coordinate (Y) at (0, -14);
  \coordinate (YY) at (0, 14);
\draw[thin,->] (Y) -- (YY);
\node at (-0.3,14) {\small $y$};

\coordinate (U) at (-3,12);

\coordinate (UU) at (3,12);
\draw[black,thin] (UU) -- (U);
\coordinate (V) at (-3,-12); \coordinate (VV) at (3,-12);
\draw[black,thin] (V) -- (U);

\draw[black,thin](VV)--(UU);
\draw[black,thin] (V)--(VV);

\node[thin] at (1.5,12.5){$R+r/2$};

\node[thin] at (-0.4,-0.4) {$O$};
\filldraw[black,thin] (3,0) circle (0.05);
\node[black,thin] at (2.3,-0.47)
{\small$\frac{\sqrt{3}}{2}r$};
 \filldraw (1,0) circle (0.05);
\node[black] at (1.2,-0.3){\small $x$};
\filldraw[black,thin] (-3,0) circle (0.05);
\node[black,thin] at (-2.7,-0.47)
{\small $ -\frac{\sqrt{3}}{2}r$};

\draw[black,thin] (1,-12) -- (1,12);
\coordinate (T) at (0,3.5); \coordinate (TT) at (1, 3.5);
\filldraw[black,thin] (1,3.5) circle(0.1);
\node[black] at (-0.3,3.5) {\small $t$};
\coordinate (S) at (0,5.75);
\coordinate (SS) at (1,5.75);
\filldraw[black,thin] (1,5.75) circle(0.1);
\draw[black, ultra thick] (TT) --  (SS);
\draw[black, dashed] (S) --  (SS);
\draw[black, dashed] (T)-- (TT);
\node[black] at (-0.8,5.75)
{\small $t+\frac{r}{2}$};

\node[black] at (1.35,8) {$I_x$};
\filldraw[black,thin] (0,12) circle(0.05);

\node at (1,-15) {~};

\end{tikzpicture}
\caption{Looking for a suitable interval} \label{F:equi}
\end{figure}
The interval of the
corresponding values of $x$ is $\ls{-\frac{\sqrt{3}}{2}r,
\frac{\sqrt{3}}{2}r}$. 
Applying the Fubini theorem
(e.g. Theorem 14.1  in \cite{Di16}) to the characteristic function of the set of non-differentiability
points of $T$ in the $x$-axis-symmetric rectangle of height $2R+r$
have measure $0$ for almost all $x$. 

Also, the intersections of $I_x$ with the blue and yellow
sets are measurable for almost all $x$. Hence,  we can pick $x$ for
which the ``vertical'' line segment is blue or yellow almost
everywhere and blue-yellow pieces are measurable. Consider  a
moving subsegment of length $r/2$ along this
$I_x$ line segment. 
We claim that there is a position at which the measure of yellow
points on this segment is exactly $r/4$. This can be done as
follows. For $0\le t\le R$, consider a line segment $[t,t+\frac r2]$ 
and the integral $F(t):=\int_{t}^{t+\frac r2} c(s) ds$, 
where
$c(s)=-1$ if $(x,s)$ is blue and $c(s)=1$ if $(x,s)$ is
yellow. 
Then $F(t)$ is a continuous function which varies from
$-r/2$ to $r/2$
 as $t$ ranges from $0$ to $R$. 
This is because for $s\in [0, r/2]$  we have $\nm{(x,s)} \le r$ and therefore 
$DT(x,s)Y = Y_1$ and $c(s)=-1$, while for  $s\in [R,R+r/2]$  we have 
$\nm{(x,s)} \ge R$ and $DT(x,s)Y = Y_2$ and $c(s)=1$. 
Therefore $F$ attains value $0$ for some $0\le t_0 \le R$.

The argument will be completed in the following
way.   Since $T$ is a Lipschitz function,  the norm equivalence \eqref{E:NormEquiv} implies that
each one of its four components is also Lipschitz.
Since the Fundamental Theorem of Calculus holds for absolutely continuous functions (e.g. Proposition 7.2 in \cite{Di16}),
it holds for Lipschitz functions.
We use $[t_0, t_0+r/2]$ to parameterize the interval
above (with the measure of blue set equal to the measure of the  yellow set equal to $r/4$) as
\[t_0\le t \le t_0+r/2 \to p(t) = (x,t).\]
Let $a= p(t_0)$ be the bottom endpoint and
$b = p(t_0+r/2)$ be the top endpoint of the interval.
Denote by $I$ the set of those $t \in [t_0, t_0+r/2]$ for which $T$ is differentiable at $p(t)$.
$I$ is not necessarily an interval but it has $1$-dimensional Lebesgue measure $|I|=r/2$.
Applying the Fundamental Theorem of Calculus to $T$,  one obtains:
\begin{equation}
\label{E:FTCLip}
T(b)-T(a)=\int_I DT(p(t))u \ dt.
\end{equation}

We claim  that the $X$-norm  of this integral cannot be
$(1+\ep)$-equivalent to $\|b-a\| =r/2$. 
Splitting the integral as
\begin{equation}
\label{E:Int|b-a|}
\int_I DT(p(t))u \ dt=  \int_{I_1}
DT(p(t))u \ dt
  +\int_{I_2} DT(p(t))u\ dt,
\end{equation}
where on
the right-hand side we consider integrals over values $t \in I_1$ for which $p(t)$ is in the blue set
and values $t \in I_2$ for which $p(t)$ is in the yellow set. Note that $I_1$ and $I_2$ are
measurable subsets of
$I$  and that
$|I_1|=|I_2|=r/4$ by the previous step.
Now,  we estimate  the norm of
the integral in \eqref{E:Int|b-a|} from above.

With the notation
$T=(T_1,T_2)$, one has:
\[DT(p(t))u =
DT_1(p(t))u +DT_2(p(t))u
\in Y_1 \oplus Y_2.\]

For $t\in I_1$, the definition of $I_1$  implies that
\[DT(p(t))u \in DT(p(t))Y \ \ \mbox{ with } \ \
\Omega\lp{DT(p(t))Y,Y_1}<\gamma.\]

Further, we need the following

\begin{observation}\label{O:UCLargeDis} For any vector $y=(y_1, y_2) \in Z$ for
some $2$-dimensional subspace $Z$ of $X$ for which
$\Omega(Z,Y_1)\le \gamma$, it holds $\nm{y_2} \le \gamma
\nm{y}_X$. Similarly if $\Omega(Z,Y_2)\le \gamma$ then $\nm{y_1}
\le \gamma \nm{y}_X$.
\end{observation}

\begin{proof} Assume that $y=(y_1,y_2)
\in Z$, where $Z$ is  a $2$-dimensional subspace of $X$ such that $
\Omega(Z,Y_1)\le \gamma$.  This implies that $d_X(y,Y_1)\le \gamma
\nm{y}_X$. Let $w$ be a vector in $Y_1$ such that
\[ \nm{y-w}_X = d_X(y,Y_1).\]
Then,
\[ \nm{y_2} = \nm{y-y_1} \le \nm{y-w}_2
\stackrel{\eqref{E:NormEquiv}}{\le} \nm{y-w}_X \le \gamma
\nm{y}_X. \]
\end{proof}

Using this observation, we obtain that  for every $t\in I_1$, 
\[\nm{DT_2(p(t))u} \le \gamma
\nm{DT(p(t))u}_X \le \gamma(1+\ve).\]
Similarly, for every $t\in I_2$, we have:
\[DT(p(t))u \in DT(p(t))Y \ \ \mbox{ with } \ \
\Omega\lp{DT(p(t))Y,Y_2}<\gamma,\]
and hence
\[\nm{DT_1(p(t))u} \le \gamma
\nm{D_T(p(t))u}_X \le \gamma(1+\ve).\]
Re-write \eqref{E:FTCLip} and \eqref{E:Int|b-a|} as
\[ \begin{aligned}
T(b)-T(a) = &
\lp{\int_{I_1}
DT_1(p(t))u \ dt +\int_{I_2}
DT_1(p(t))u \ dt} \\
 & +  \lp{\int_{I_1}
DT_2(p(t))u \ dt +\int_{I_2}
DT_2(p(t))u \ dt}.
\end{aligned}\]
The first parenthesis contains a vector $v_1$ in $Y_1$ with norm bounded
by
\[ \begin{aligned}
 \nm{v_1} = & \nm{\int_{I_1}
DT_1(p(t))u \ dt +\int_{I_2}
DT_1(p(t))u \ dt} \\
& \le \int_{I_1}
\nm{DT_1(p(t))u} \ dt +\int_{I_2}
\nm{DT_1(p(t))u} \ dt \\
 & \le \int_{I_1} (1+\ve) \ dt +  \int_{I_2} \gamma(1+\ve) \ dt
= (1+\gamma)(1+\ve)\frac r4.
\end{aligned} \]
Similarly, the second parenthesis is a vector $v_2$ in $Y_2$ with
the same upper bound for the norm.

Therefore,
\[ \nm{T(b)-T(a)}_X =
\nm{v_1+v_2}_X \stackrel{\eqref{E:NormEquiv}}{\le}
\frac{1}{1-\delta^2/2} \nm{v_1+v_2}_2 \le
\frac{1}{1-\delta^2/2}\sqrt{2} (1+\gamma)(1+\ve)\frac r4,
\]
where the last inequality follows from the Pythagorean Theorem and the estimates on the norms of $v_1$ and $v_2$.

Since
\[ \frac{1}{1+\ve}\frac r2 =
\frac{1}{1+\ve}\nm{b-a} \le
\nm{T(b)-T(a)}_X,\]
we obtain
\[\frac{1}{1+\ve}\frac r2  \le
\frac{1}{1-\delta^2/2}\sqrt{2} (1+\gamma)(1+\ve)\frac r4.
\]
Thus,
\[1-\frac{\delta^2}{2} \le
\frac{(1+\gamma)(1+\ve)^2}{\sqrt{2}}.
\]

As $\ve$ was chosen strictly less than $\gamma$, we derive:
\[ 1-\frac{\delta^2}{2} < \frac{(1+\gamma)^3}{\sqrt{2}}.\]
However, this contradicts \eqref{E:pick_del} and, thus,  it contradicts the existence of the function $T$
with the required properties.
\end{proof}

\section*{Acknowledgement}The second-named  author gratefully acknowledges the support of Atilim university as this work was mostly conducted while she was on research leave supported by Atilim University. Also, she expresses her sincere gratitude to professor G. M. Feldman (B.Verkin Institute for Low Temperature Physics and Engineering
of the National Academy of Sciences of Ukraine) for his invitation
the Department of Function Theory for this research leave and his
help during her stay at the Department. The third-named author
gratefully acknowledges the support by the National Science
Foundation grant NSF DMS-1953773.

We thank the anonymous referee for many helpful comments.

\end{large}

\renewcommand{\refname}{

\textsc{Department of Mathematics and Computer Science, St. John's
University, 8000 Utopia Parkway, Queens, NY 11439, USA} \par
  \textit{E-mail address}: \texttt{catrinaf@stjohns.edu} \par
  \medskip

\textsc{Department of Mathematics, Atilim University, 06830
Incek,\\ Ankara, TURKEY} \par \textit{E-mail address}:
\texttt{sofia.ostrovska@atilim.edu.tr}\par\medskip

\textsc{Department of Mathematics and Computer Science, St. John's
University, 8000 Utopia Parkway, Queens, NY 11439, USA} \par
  \textit{E-mail address}: \texttt{ostrovsm@stjohns.edu} \par


\section*{References}}

\begin{thebibliography}{WW}

\bibitem{AGM15} S.~Artstein-Avidan, A.~Giannopoulos, V.\,D.~Milman, {\it Asymptotic geometric analysis.} Part I. Mathematical Surveys and
Monographs, 202. American Mathematical Society, Providence, RI,
2015.

\bibitem{AGM21} S.~Artstein-Avidan, A.~Giannopoulos, V.\,D.~Milman, {\it Asymptotic geometric analysis.} Part II. Mathematical Surveys and
Monographs, 261. American Mathematical Society, Providence, RI,
2021.

\bibitem{BBM06} Y.~Bartal, B.~Bollob\'as, M.~Mendel, Ramsey-type
theorems for metric spaces with applications to online problems.
J. Comput. System Sci. 72 (2006), no. 5, 890--921.

\bibitem{BLMN05} Y.~Bartal,
N.~Linial, M.~Mendel, A.~Naor, On metric Ramsey-type phenomena,
{\it Annals of Math.}, {\bf 162} (2005), 643--709.

\bibitem{BL08} F.~Baudier, G.~Lancien, Embeddings of locally
finite metric spaces into Banach spaces, {\it Proc. Amer. Math.
Soc.}, {\bf 136} (2008), 1029--1033.

\bibitem{BLS18} F.~Baudier, G.~Lancien, Th.~Schlumprecht,
The coarse geometry of Tsirelson's space and applications. {\it J.
Amer. Math. Soc.} {\bf 31} (2018), no. 3, 699--717.

\bibitem{BL00} Y.~Benyamini, J.~Lindenstrauss, {\it Geometric nonlinear
functional analysis}. Vol. {\bf 1}. American Mathematical Society
Colloquium Publications, {\bf 48}. American Mathematical Society,
Providence, RI, 2000.

\bibitem{BFM86} J.~Bourgain, T.~Figiel, V.~Milman, On Hilbertian subsets of
finite metric spaces. Israel J. Math. 55 (1986), no. 2, 147--152.

\bibitem{BS07} S.~Buyalo, V.~Schroeder, {\it Elements of
asymptotic geometry.} EMS Monographs in Mathematics. European
Mathematical Society (EMS), Z\"urich, 2007.

\bibitem{DG03} M.~Dadarlat, E.~Guentner,
Constructions preserving Hilbert space uniform embeddability of
discrete groups. {\it Trans. Amer. Math. Soc.} {\bf 355} (2003),
no. 8, 3253--3275.

\bibitem{Di16} E. DiBenedetto,
{\it Real analysis.} 2nd edition.
Birkh\"{a}user Advanced Texts. Basler Lehrb\"{u}cher. New York, NY:
Birkh\"{a}user/Springer, 2016.

\bibitem{Dvo59} A.~Dvoretzky, A theorem on convex bodies and applications to Banach spaces, {\it
Proc. Nat. Acad. Sci. U.S.A.}, {\bf 45} (1959) 223--226; erratum,
1554.

\bibitem{Dvo61} A.~Dvoretzky, Some results on convex bodies and Banach spaces, in:
{\it Proc. Internat. Sympos. Linear Spaces} (Jerusalem, 1960),
pp.~123--160, Jerusalem Academic Press, Jerusalem; Pergamon,
Oxford, 1961.

\bibitem{Gor88} Y.~Gordon, Gaussian processes and almost spherical
sections of convex bodies. {\it Ann. Probab.} {\bf 16} (1988), no.
1, 180--188.

\bibitem{Gro53} A.~Grothendieck, Sur certaines classes de suites dans les espaces de Banach et le th\'eor\`eme de Dvoretzky-Rogers.
{\it Bol. Soc. Mat. S\~ao Paulo} {\bf 8} (1953), 81--110.

\bibitem{Gru07}
P.\,M.~Gruber, {\it Convex and discrete geometry}. Grundlehren der
mathematischen Wissenschaften [Fundamental Principles of
Mathematical Sciences], 336. Springer, Berlin, 2007.

\bibitem{JL01} W.\,B.~Johnson,  J.~Lindenstrauss, Basic
concepts in the geometry of Banach spaces, in: {\it Handbook of
the geometry of Banach spaces} (W.\,B.~Johnson and
J.~Lindenstrauss, Eds.) Vol. {\bf 1},  Elsevier, Amsterdam, 2001,
pp.~1--84.

\bibitem{KO18} J.~Kilbane, M.\,I.~Ostrovskii, There is no finitely
isometric Krivine's theorem. {\it Houston J. Math.} {\bf 44}
(2018), No. 1, 309--317.

\bibitem{LM75} D.\,G.~Larman, P.~Mani, Almost ellipsoidal sections and projections of convex bodies. {\it Math. Proc. Cambridge Philos. Soc.}
{\bf 77} (1975), 529--546.

\bibitem{LT77} J.~Lindenstrauss, L.~Tzafriri,
{\it Classical Banach spaces. {\bf I}. Sequence spaces}.
Ergebnisse der Mathematik und ihrer Grenzgebiete, Vol. {\bf 92}.
Springer-Verlag, Berlin-New York, 1977.

\bibitem{LT79} J.~Lindenstrauss, L.~Tzafriri, {\it Classical Banach
spaces. {\bf II}. Function spaces}. Ergebnisse der Mathematik und
ihrer Grenzgebiete [Results in Mathematics and Related Areas],
{\bf 97}. Springer-Verlag, Berlin-New York, 1979.

\bibitem{Mat02} J.~Matou\v sek, {\it Lectures on discrete geometry}. Graduate Texts in Mathematics, {\bf 212}. Springer-Verlag, New York, 2002.

\bibitem{MN07} M.~Mendel, A.~Naor, Ramsey partitions and proximity data structures. {\it J. Eur. Math.
Soc.} (JEMS) {\bf 9} (2007), no. 2, 253--275.

\bibitem{MN13} M.~Mendel, A.~Naor, Ultrametric subsets with large
Hausdorff dimension. {\it Invent. Math.} {\bf 192} (2013), no. 1,
1--54.

\bibitem{Mil71} V.\,D.~Milman,
A new proof of A. Dvoretzky's theorem on cross-sections of convex
bodies. (Russian) Funkcional. Anal. i Prilo\v zen. 5 (1971), no.
4, 28--37.

\bibitem{MS86} V.\,D.~Milman, G.~Schechtman, {\it Asymptotic theory of
finite-dimensional normed spaces.}
 With an appendix by M. Gromov. Lecture Notes in Mathematics, 1200. Springer-Verlag, Berlin, 1986.

\bibitem{Nao12} A.~Naor,
An introduction to the Ribe program, {\it Jpn. J. Math.}, {\bf 7}
(2012), no.~2, 167--233.

\bibitem{NT12}  A.~Naor, T.~Tao, Scale-oblivious metric fragmentation and
the nonlinear Dvoretzky theorem. Israel J. Math. 192 (2012), no.
1, 489--504.

\bibitem{Now06} P.\,W.~Nowak,
On coarse embeddability into $\ell\sb p$-spaces and a conjecture
of Dranishnikov, {\it Fund. Math.}, {\bf 189} (2006), no. 2,
111--116.

\bibitem{NY12} P.\,W.~Nowak, G.~Yu, {\it
Large scale geometry}.
EMS Textbooks in Mathematics. European Mathematical Society (EMS), Z\"urich, 2012.

\bibitem{OS94} E.~Odell, T.~Schlumprecht,
The distortion problem, {\it Acta Math.}, {\bf 173} (1994), no. 2,
259--281.

\bibitem{OO19} S.~Ostrovska, M.\,I.~Ostrovskii, Distortion in the finite determination result for embeddings
of locally finite metric spaces into Banach spaces, {\it Glasgow
Math. J.}, {\bf 61} (2019) 33--47.

\bibitem{OO19b} S.~Ostrovska, M.\,I.~Ostrovskii, On embeddings
of locally finite metric spaces into $\ell_p$. {\it J. Math. Anal.
Appl.} {\bf 474} (2019), no. 1, 666--673.

\bibitem{Ost94} M.\,I.~Ostrovskii, Topologies on the set of all subspaces of a Banach space and
related questions of Banach space geometry, {\it Quaestiones
Math.}, {\bf 17} (1994), no. 3, 259--319.

\bibitem{Ost06} M.\,I.~Ostrovskii, On comparison of the coarse embeddability into a Hilbert space and
into other Banach spaces, {\it unpublished manuscript}, 2006,
available at {\tt http://facpub.stjohns.edu/ostrovsm}

\bibitem{Ost09} M.\,I.~Ostrovskii, Coarse embeddability into Banach spaces, {\it Topology Proc.},
{\bf 33} (2009), 163--183.

\bibitem{Ost13} M.\,I.~Ostrovskii, {\it Metric Embeddings: Bilipschitz and Coarse Embeddings into Banach
Spaces}, de Gruyter Studies in Mathematics, {\bf 49}. Walter de
Gruyter \&\ Co., Berlin, 2013.

\bibitem{Ost15} M.\,I.~Ostrovskii, Isometric embeddings of finite subsets of  $\ell_2$  into
infinite-dimensional Banach spaces, {\tt
https://mathoverflow.net/questions/221181/}

\bibitem{PV18} G.~Paouris, P.~Valettas, Dichotomies, structure, and concentration in normed spaces. {\it Adv. Math.} {\bf 332} (2018), 438--464.

\bibitem{Sch06} G.~Schechtman, Two observations regarding embedding subsets
of Euclidean spaces in normed spaces. Adv. Math. 200 (2006), no.
1, 125--135.

\bibitem{Tsi74} B.\,S.~Tsirelson, It is impossible to imbed $\ell_{p}$ of $c_{0}$ into an arbitrary Banach
space, {\it Functional Anal. Appl.}, {\bf 8} (1974), 138--141.

\end{thebibliography}
\end{document}